\documentclass[12pt]{article}
\usepackage[left=1in,top=1in,right=1in,bottom=1in,nohead]{geometry}
\usepackage{hyperref, amssymb, mathtools, amsmath, color, bbm, amsthm}
\usepackage{makeidx}
\usepackage{tikz}
\usetikzlibrary{backgrounds,intersections}
\usepackage{caption}
\usepackage{listings}
\usepackage{array}
\usepackage{wrapfig}
\usepackage{multirow}
\usepackage{tabularx}
\usepackage{bigints}

\let\realverbatim=\verbatim
\let\realendverbatim=\endverbatim
\renewcommand\verbatim{\par\addvspace{6pt plus 2pt minus 1pt}\realverbatim}
\renewcommand\endverbatim{\realendverbatim\addvspace{6pt plus 2pt minus 1pt}}

\makeatletter

\newcommand{\xrightleftarrows}[3]{\;\mathop{\overset{\raise -1pt \hbox{\scriptsize $#1$}}{\substack{\xrightarrow{\rule{#3}{0pt}} \\[-.9ex] \xleftarrow{\rule{#3}{0pt}\rule{0pt}{1pt}}}}}_{#2}\;}

\newcolumntype{M}[1]{>{\centering\arraybackslash}m{#1}}

\definecolor{mygreen}{RGB}{28,172,0} 
\definecolor{mylilas}{RGB}{170,55,241}

\newsavebox{\astrutbox}
\sbox{\astrutbox}{\rule[-5pt]{0pt}{20pt}}

\newcommand\Sp{\mathcal{S}}
\newcommand\Cx{\mathcal{C}}
\newcommand\Rx{\mathcal{R}}
\newcommand\Ki{\mathcal{K}}
\newcommand\ZZ{\mathbb{Z}}
\newcommand\vvec {\mathbf}

\newtheorem{theorem}{Theorem}[section]

\newtheorem{lemma}{Lemma}[section]
\newtheorem{corollary}{Corollary}[section]
\newtheorem{proposition}{Proposition}[section]

\theoremstyle{definition}
\newtheorem{definition}[theorem]{Definition}

\newtheorem{remark}{Remark}[section]
\numberwithin{equation}{section}

\newtheorem{example}{Example}[section]

\title{Identifiability 
of Stochastically Modelled \\
Reaction Networks}

\author{German Enciso\thanks{Department of Mathematics, University of California, Irvine, USA; e-mail: enciso@uci.edu} 
\and
Radek Erban\thanks{Mathematical Institute, University of Oxford, 
Radcliffe Observatory Quarter, Woodstock Road, Oxford, OX2 6GG, 
United Kingdom; e-mail: erban@maths.ox.ac.uk}, 
\and  
Jinsu Kim\thanks{Department of Mathematics, University of California, 
Irvine, USA; e-mail: jinsu.kim@uci.edu}}

\date{}

\begin{document}

\label{firstpage}
\maketitle

\begin{abstract}
\noindent
Chemical reaction networks describe interactions between biochemical species.
Once an underlying reaction network is given for a biochemical system, the
system dynamics can be modelled with various mathematical frameworks such as
continuous time Markov processes. In this manuscript, the identifiability of
the underlying network structure with a given stochastic system dynamics is
studied. It is shown that some data types related to the associated stochastic
dynamics can uniquely identify the underlying network structure as well as the
system parameters. The accuracy of the presented network inference is
investigated when given dynamical data is obtained via stochastic simulations.
\end{abstract}

\section{Introduction}
To study the properties and dynamics of a system of reacting biochemical
species, a network representation is often used to describe the interactions
between the chemical species involved. A reaction network represents the 
system behaviour with reactions (directed edges) between complexes
(nodes)~\cite{Craciun:2006:MEC,Feinberg:2019:FCR}. Each reaction in a reaction
network indicates loss or gain of the amount of the corresponding chemical
species. Systems of ordinary differential equations (ODEs) are traditionally
used for modelling the time evolution of concentrations of chemical species 
in reaction network theory~\cite{Feinberg:1989:NSC,Angeli:2009:TCR}. Since
biochemical systems may contain chemical species with low copy numbers,
stochastic approaches are often used for modelling their
behaviour~\cite{Erban:2020:SMR}. Stochastic models of homogeneous 
(space independent) chemical reaction networks are written as continuous 
time discrete space Markov chains~\cite{Anderson:2011:DAB,Anderson:2015:SAB}. 

In some applications, the underlying network structure may be unknown 
but information on the associated dynamics is
given~\cite{Craciun:2008:ICR,Komorowski:2011:SRA}. The main focus of this 
paper is to identify the unknown network structure of a stochastic 
reaction system by using dynamical information. Identifiability of 
reaction systems has been studied under deterministic ODE modeling by 
Craciun and Pantea~\cite{Craciun:2008:ICR} and 
Szederk\' enyi et al~ \cite{szederkenyi2011inference}. They present 
examples of reaction systems that admit the same deterministic dynamical 
system but have different network structure and parameters. 
In Figure~\ref{figure1}, we illustrate this lack 
of identifiability using two simple reaction systems. 
They both include one chemical species $X_1$, which is subject 
to two chemical reactions 
\begin{eqnarray}
&&\mbox{first reaction system:}
\hskip 2.48cm
\emptyset \mathop{\longrightarrow}^{1} X_1 \, ,
\qquad \hskip 3mm
X_1 \mathop{\longrightarrow}^{1} \emptyset \, ,
\hskip 12mm
\label{firssystem} \\
&&\mbox{second reaction system:}
\hskip 2cm
\emptyset  \mathop{\longrightarrow}^{1/4} 4 X_1 \, , 
\qquad \hskip 1mm
X_1 \mathop{\longrightarrow}^{1} \emptyset \, .
\hskip 12mm
\label{secondsystem}
\end{eqnarray}
Denoting $x_1(t)$ the concentration of the chemical species $X_1$ and 
using mass-action deterministic description, the time evolution of both
reaction systems~(\ref{firssystem}) and (\ref{secondsystem}) is described 
by the same ODE
\begin{equation}
\frac{\mbox{d} x_1}{\mbox{d}t}
=
1 - x_1.
\label{ODEsystem}
\end{equation}
Solving the ODE~(\ref{ODEsystem}) with the initial condition $x_1(0)=0$, 
we obtain $x_1(t) = 1 - \exp[-t]$, which is plotted in Figure~\ref{figure1}(a).
Since both reaction systems~(\ref{firssystem}) and~(\ref{secondsystem}) contain
only reactions of zero and first order, we can analytically solve the 
chemical master equation corresponding to the stochastic
model~\cite{Gadgil:2005:SAF,Jahnke:2007:SCM}. We obtain that the mean number 
of molecules, $\langle X_1 \rangle$, is for both systems given as a solution 
of the ODE system~(\ref{ODEsystem}). In the case of the first reaction
system~(\ref{firssystem}), $X_1$ is Poisson distributed at every time 
$t$~\cite{Erban:2020:SMR,Jahnke:2007:SCM}. Therefore, the variance 
$\langle X_1^2 \rangle - \langle X_1 \rangle^2$ is equal to the mean 
$\langle X_1 \rangle = 1 - \exp[-t].$ In Figure~\ref{figure1}(b), we show 
that it differs from the variance obtained using the second reaction
system~(\ref{secondsystem}), which is given as
$
\langle X_1^2 \rangle - \langle X_1 \rangle^2 
= 
(5 - 2 \exp[-t] - 3 \exp[-2t])/2
$.

\begin{figure}
\leftline{(a) \hskip 7.9cm (b)}
\centerline{\hskip 8mm \includegraphics[height=6.2cm]{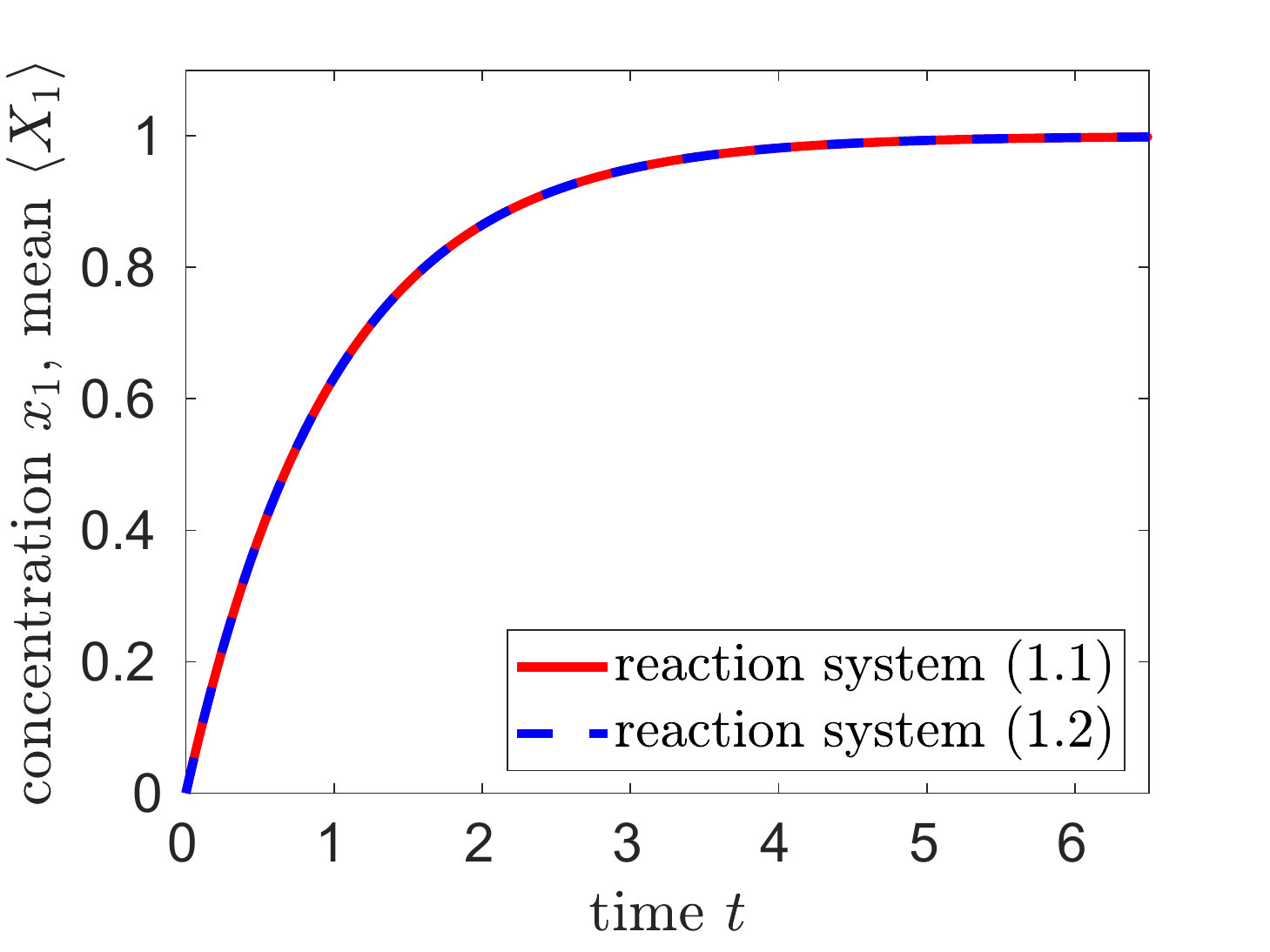}
$\;\;$\includegraphics[height=6.2cm]{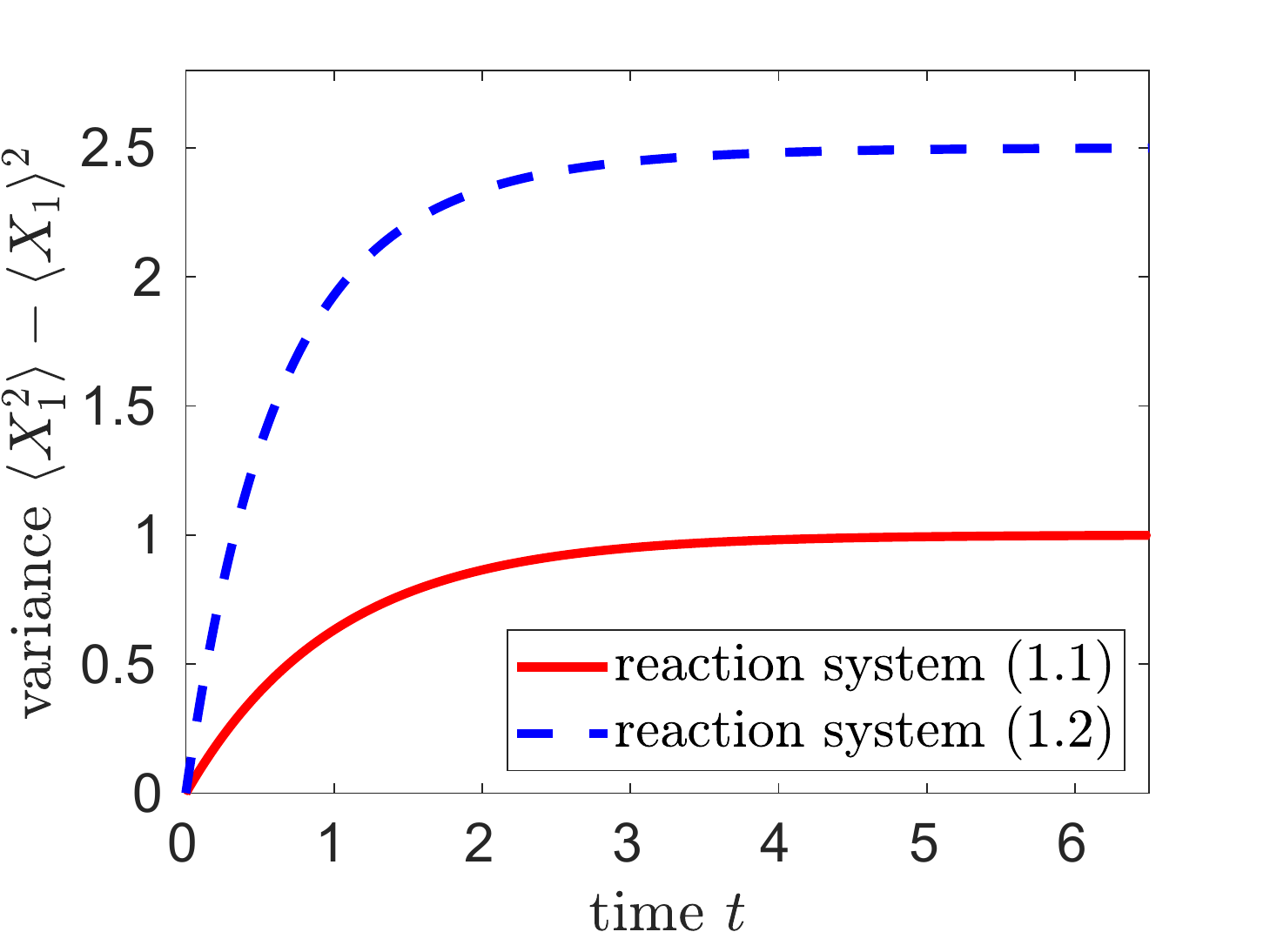}}
\caption{(a) {\it The solution of ODE~$(\ref{ODEsystem})$ with 
initial condition $x_1(0)=0$.}
(b) {\it Variance of the number of molecules of chemical species $X_1$ for 
the chemical system $(\ref{firssystem})$ (red solid line) and 
the chemical system $(\ref{secondsystem})$ (blue dashed line).}}
\label{figure1}
\end{figure}

Our example illustrates that the dynamics obtained by the ODE
model~(\ref{ODEsystem}) cannot be used to distinguish between reaction
systems~(\ref{firssystem}) and~(\ref{secondsystem}) and the reaction network 
is therefore not identifiable in the deterministic context. However, since
their stochastic models do differ (as shown in Figure~\ref{figure1}(b)), we
have potential to use the stochastic data to distinguish between the reaction
systems~(\ref{firssystem}) and~(\ref{secondsystem}). This peculiar behaviour 
is not restricted to our illustrative example. 
Plesa et al~\cite{Plesa:2018:NCM} showed that any reaction network can 
be redesigned in such a way that the deterministic dynamics are preserved,
while the controllable state-dependent noise is introduced into the stochastic
dynamics. In this way, one can systematically obtain a family of reaction
networks, which have qualitatively different stochastic dynamics, but they 
are described by the same deterministic model~\cite{Plesa:2018:NCM}. 
In applications, the long-term dynamics of some gene regulatory networks
(involving multiple time-scales) can consist of a unique attractor at 
the deterministic level (unistability), while the long-term probability
distribution at the stochastic level may display multiple maxima
(multimodality)~\cite{Duncan:2015:NIM,Plesa:2019:NIM}.

In this paper, we explore how the discrete nature of the associated 
mass-action stochastic system can help uncover the underlying reaction 
network. For a given continuous time Markov chain, we quantify the amount 
of transition rate information needed to uniquely identify the underlying
network and the system parameters. For practical implementation of network
inference, the presented approach can be used to infer the underlying 
reaction network with transition data obtained from stochastic simulations. 
The accuracy of this network inference idea is also investigated.

For each reaction, the reaction intensity, which determines the likelihood 
of firing the reaction, is proportional to a positive constant, so-called 
a rate constant such as numbers $1/4$ and $1$ in our illustrative reaction
system~(\ref{secondsystem}). The rate constants can alter the system 
behaviour significantly and correspond to qualitative differences between
deterministic and stochastic descriptions, for example, for systems close 
to bifurcations of deterministic ODEs~\cite{Erban:2009:ASC,Plesa:2017:TMS}.
When the reaction network topology is given, the rate constants often need 
to be estimated as missing parameters. Numerous different statistical and
mathematical techniques have been employed in the literature for parameter
estimation using dynamical data, such as information
theory~\cite{Komorowski:2011:SRA}, Bayesian 
statistics~\cite{craciun2013statistical,golightly2006bayesian,catanach2020bayesian,warne2019simulation}, 
system identification theory~\cite{walter1997identification}, machine 
learning~\cite{baldi2001bioinformatics} and tensor-structured parametric
analysis~\cite{Liao:2015:TMP}. 

In addition to parameter estimation, the underlying network topology 
is also often unknown or only partially known. There have also been 
a number of methods developed in the literature to infer network 
information~\cite{chattopadhyay2013inverse,Langary:2019:ICR,Wang:2019:IRN}.
For instance, Wang et al.~\cite{Wang:2019:IRN} study deterministic network
inference using multiplex flow cytometry experimental data and toric 
systems theory. Chattopadhyay et al.~\cite{chattopadhyay2013inverse} 
proposed a novel inference method for stochastic reaction systems with 
convex polytopes, which are formed by combinations of reaction vectors 
captured within a short time window. Other papers focus on statistical 
information and Bayesian analysis to infer networks of correlations 
among species~\cite{Langary:2019:ICR,gupta2014comparison, Loskot:2019:CRM, markowetz2007inferring, villaverde2014reverse}, but, to our knowledge, 
there is no previous work 
that characterizes when the transition data of a stochastic system can 
be used to completely identify the underlying reaction network. 

For the validity of such parameter estimation tools and network inference
algorithms, we consider identifiability of a reaction system. The underlying
network structure of a dynamical system may not be uniquely identified if 
prior information is partially given. For example, when a continuous time 
Markov chain is restricted to a subset of the state space because of a
conservation law, this stochastic system can be associated with two 
different reaction networks, as illustrated in Example~\ref{ex:non unique1}.
In Section~\ref{sec3}, we prove that the network topology and the system
parameters can be uniquely identified provided that we have full dynamic
information in a sufficiently large finite region of the state space.

To formulate our results, we begin with introducing our notation in
Section~\ref{sec2}. In Section~\ref{sec3}, we present the main algorithm 
that uses the transition rates of a given continuous time Markov chain 
to infer the underlying network structure and parameters. In Section 4, 
we show that a general continuous time Markov chain with 
polynomial transition rates can be identifiable as a mass-action reaction 
system. In Section 5, with given stochastic dynamical information about 
the transition rates, we investigate how accurately the underlying network
structure and system parameters can be identified.

\section{Notation and terminology}
\label{sec2}

In this section, we introduce our notation and basic definitions that are 
used throughout the rest of our manuscript.

\subsection{Reaction networks}\label{subsec:crn}
A reaction network $(\Sp,\Cx,\Rx)$ consists of \emph{species}, \emph{complexes} and \emph{reactions}. Each reaction is of the form
\begin{equation}
\label{eq:reaction}
\sum_{i=1}^d y_i X_i 
\;\, \longrightarrow \;\, 
\sum_{i=1}^d y'_i X_i,
\end{equation}
where $X_i,$ $i=1,2,\dots,d$, are species, and linear combinations 
$\sum_{i=1}^d y_i X_i$ and $\sum_{i=1}^d y'_i X_i$ of species are 
complexes. We interchangeably denote by 
${\mathbf y}=(y_1,y_2,\dots,y_d)$ a complex $\sum_{i=1}^d y_i X_i$. 
In the same way, we denote by ${\mathbf y}\to {\mathbf y}'$ the 
reaction~\eqref{eq:reaction}. We denote by $\Sp,$ $\Cx$, and 
$\Rx$ the collections of species, complexes, and reactions,
respectively, in the reaction network~$(\Sp,\Cx,\Rx)$.

\begin{example}
\label{example1}
The typical enzyme-substrate system can be described with a reaction network
\begin{align*}
X_1+X_2
\xrightleftarrows{\kappa_1}{\kappa_2}{4 mm}
X_3 \xrightarrow{\;\kappa_3\;} 
X_1+X_4,
\end{align*}
where the species $X_1,$ $X_2,$ $X_3$ and $X_4$ represent the enzyme, 
substrate, enzyme-substrate complex and product, respectively. For this 
system, we have~$\Sp=\{X_1, \, X_2, \, X_3, \, X_4 \}$, 
$\Cx=\{X_1 + X_2, \, X_3, \, X_1 + X_4\}$ and 
$\Rx=\{X_1 + X_2 \to X_3, \, X_3 \to X_1 + X_2, \, X_3 \to X_1 + X_4 \}$. 
Each reaction in $\Rx$ is associated with the corresponding rate constant 
$\kappa_1,$ $\kappa_2$ and $\kappa_3$.
\end{example}

\noindent
The time evolution of the concentration of species $X_i \in \Sp$ is described 
with a system of ODEs as
\begin{align*}
\frac{\mbox{d}{\mathbf x}}{\mbox{d}t}(t)
=\sum_{{\mathbf y}\to {\mathbf y}' \in \Rx } 
f_{{\mathbf y} \to {\mathbf y}'}({\mathbf x}(t)) 
({\mathbf y}'-{\mathbf y}),
\end{align*}
where $f_{{\mathbf y}\to {\mathbf y}'}$ are positive functions 
representing the weight of the reaction ${\mathbf y} \to {\mathbf y}'$ 
at each state. Considering \emph{mass-action kinetics}, we have
$$
f_{{\mathbf y}\to {\mathbf y}'}(\mathbf x)
=\kappa_{{\mathbf y}\to {\mathbf y}'}\vvec x^{\mathbf y},
$$
where ${\mathbf u}^{\mathbf v}=\prod_{i=1}^d u_i^{v_i}$ for 
vectors ${\mathbf u}$ and ${\mathbf v}$ with non-negative entries. 
The positive constant $\kappa_{{\mathbf y}\to {\mathbf y}'}$ forms 
the reaction rate for the reaction, and it constitutes one of the 
parameters of the reaction network. We include this reaction 
rate by placing it above the arrow of the associated
reaction ${\mathbf y}\to {\mathbf y}'$ as in Example~\ref{example1}.

\subsection{Stochastic description of reaction networks}

We model the number of molecules of each chemical species 
in a reaction network by a continuous-time Markov 
chain (CTMC) defined on the $d$-dimensional integer lattice 
\begin{equation}
\mathbb Z^d_{\ge 0}
=
\left\{\vvec x \in \mathbb Z^d
\; | \; x_i \ge 0 \mbox{ for } i = 1,\, 2, \, \dots, \, d 
\right\}.
\label{defZdlat}
\end{equation}
Denoting ${\mathbf X}(t) = [X_1(t), \, X_2(t), \, \dots, \, X_d(t)]$
the number of molecules in reaction network $(\Sp,\Cx,\Rx)$, 
the corresponding 
transition rates are defined as
\begin{equation*}
P({\mathbf X}(t+\Delta t)
=
\vvec x+ {\mathbf z} \ | \ {\mathbf X}(t)=\vvec x)
\, = \,
\sum_{\substack{{{\mathbf y}\to {\mathbf y}'}\in \Rx 
\\ 
{\mathbf y}'-{\mathbf y}\,=\,{\mathbf z}}}
\lambda_{{{\mathbf y}\to {\mathbf y}'}}(\vvec x) \, \Delta t 
+ o(\Delta t),
\end{equation*}
where $o(\Delta t) \to 0$, as $\Delta t \to 0$. 
We denote by 
$\cal Z
=
\{\vvec z=\vvec y'-\vvec y: \vvec y \to \vvec y'\in \Rx\}
$ 
the set of the transition vectors of the CTMC $\vvec X(t)$. 
The function $\lambda_{{\mathbf y}\to {\mathbf y}'} \ge 0$ 
is called the \emph{intensity} of reaction ${\mathbf y}\to {\mathbf y}'$ 
and it satisfies
\begin{equation}
\label{eq:standard property}
\lambda_{{{\mathbf y}\to {\mathbf y}'}}(\vvec x)>0 \quad 
\text{if and only if $\; x_i \ge y_i \;$ for each $i = 1,\,2,\,\dots,\,d$}.
\end{equation}
We say that a reaction ${\mathbf y}\to {\mathbf y}' \in \Rx$ is \emph{turned off} 
at $\vvec x$ if $\lambda_{{\mathbf y}\to {\mathbf y}'}(\vvec x)=0$. 
Otherwise we call a reaction ${\mathbf y}\to {\mathbf y}' \in \Rx$ 
is \emph{charged} at $\vvec x$. Using (stochastic) mass-action kinetics, we define, 
for each ${{\mathbf y}\to {\mathbf y}'}\in \Rx$  
\begin{equation}
\label{eq:sto mass}
\lambda_{{\mathbf y}\to {\mathbf y}'}(\vvec x)
=\kappa_{y \to y'}\vvec x^{(y)}, \quad 
\text{where} \; \; {\mathbf u}^{({\mathbf v})}
=
\prod_{i=1}^d u_i(u_i-1)\cdots (u_i-v_i+1)
\end{equation}
for vectors ${\mathbf u}, {\mathbf v} \in \mathbb Z^d_{\ge 0}$.

Let $\Ki=\{\lambda_{{\mathbf y}\to {\mathbf y}'} : 
{\mathbf y}\to {\mathbf y}' \in \Rx\}$ be the collection of given intensities 
for a reaction network $(\Sp,\Cx,\Rx)$. Then the associated CTMC is 
fully characterized by the four tuple  $(\Sp,\Cx,\Rx,\Ki)$. Furthermore, 
since $\Sp$ and $\Cx$ can be fully determined using $\Rx$, the 
reaction system is fully characterized with $\Rx$ and $\Ki$. 
So in the rest of the paper, we let  $(\Rx,\Ki)$ represent 
both a reaction network and the associated CTMC, and we call
$(\Rx,\Ki)$ a (stochastic) reaction system. 

A reaction network $(\Rx,\Ki)$ is a subnetwork 
of another reaction network  $(\Rx',\Ki')$ if $\Rx \subset \Rx'$ 
and $\lambda_{{\mathbf y}\to {\mathbf y}'}
\equiv
\lambda'_{{\mathbf y}\to {\mathbf y}'} \in \Ki$ 
for each ${\mathbf y}\to {\mathbf y}'\in \Rx$, where 
$\lambda_{{\mathbf y}\to {\mathbf y}'}$ and 
$\lambda'_{{\mathbf y}\to {\mathbf y}'}$ are the reaction 
intensities of $(\Rx,\Ki)$ and $(\Rx',\Ki')$, respectively. 
We denote this relation as $(\Rx,\Ki)\subset (\Rx',\Ki')$. 
If two systems $(\Rx,\Ki)$ and $(\Rx',\Ki')$ are identical, 
then $(\Rx,\Ki)\subseteq (\Rx',\Ki')$ and 
$(\Rx,\Ki)\subseteq (\Rx',\Ki')$, which we shortly denote 
by $(\Rx,\Ki)= (\Rx',\Ki')$.

\subsection{Reaction order and ordering for \texorpdfstring{$\mathbb{Z}^d_{\ge 0}$}{Lg}}

As indicated in Section \ref{subsec:crn}, we use vectors to represent complexes.
Hence for $\vvec y \in \ZZ^d_{\ge 0}$ and ${\mathbf z} \in \ZZ^d$ such 
that $\vvec y + {\mathbf z} \in \ZZ^d_{\ge 0}$, we denote by 
$\vvec y \to \vvec y + {\mathbf z}$ a reaction whose source complex is 
$\vvec y=\sum_{i=1}^d y_i X_i$ and the product complex is 
$\vvec y+ {\mathbf z}=\sum_{i=1}^d(y_i + z_i)X_i$.  
For example, for $\vvec y=(1,2)^\top$ and ${\mathbf z}=(-1,1)^\top$, 
the reaction $\vvec y\to \vvec y+{\mathbf z}$ represents $X_1+2X_2\to 3X_2$.
For $\vvec v\in \mathbb Z^d_{\ge 0}$ and an integer $N$, we define 
\begin{eqnarray}
\mathbb S_{N}
&=&
\{\vvec x \in \mathbb Z^d_{\ge 0}
\; | \; \vvec x \mbox{ satisfies } \Vert \vvec x \Vert_1\le N\},
\label{defSN}
\\
\mathbb S_{\vvec v, N}
&=&\{\vvec x \in \mathbb Z^d_{\ge 0}
\; | \; \vvec x \mbox{ satisfies } \vvec v \cdot \vvec x = N\},
\label{defSvN}
\end{eqnarray}
where $\cdot$ is the canonical inner product in the Euclidean space. 
Transition rates of a given CTMC on those sets will play a critical 
role in the main algorithm of this paper for inferring an underlying 
network structure. Given two vectors ${\mathbf u} \in \ZZ^d_{\ge 0}$
and ${\mathbf v} \in \ZZ^d_{\ge 0}$, we define the
\emph{lexicographical} ordering for $\ZZ^d_{\ge 0}$ by
\begin{equation}\label{eq:lexico ordering}
\vvec u \prec \vvec v \text{ if and only if there is $k$ such that 
$u_k < v_k$ and $u_i = v_i $ for all $i <k$}.
\end{equation}
In particular, the $d$-dimensional simplex $\mathbb S_{N}$ has $n$ elements
which we enumerate in the lexicographical order, that is,
\begin{equation}
\mathbb S_N=\{\vvec x^1, \, \vvec x^2, \, \dots, \, \vvec x^{n}\},
\quad
\mbox{where}
\quad
\vvec x^i \prec \vvec x^j 
\; \mbox{ if } \; i<j
\quad
\mbox{and}
\quad
n=\binom{N+d}{d}.
\label{elemSN}
\end{equation}
A reaction $\vvec y \to \vvec y'$ is of \emph{order} $N$ 
if $\Vert \vvec y \Vert_1=N$. 
A reaction system $(\Rx,\Ki)$ is of order $N$ if the order of all
reactions in $\Rx$ is at most $N$. A reaction $\vvec y\to \vvec y'$ 
is of \emph{$\vvec v$-order} $N$ if $\vvec v \cdot \vvec y=N$. 
A reaction system $(\Rx,\Ki)$ is of $\vvec v$-order $N$ if 
the $\vvec v$-order of all reactions in $\Rx$ is at most $N$.  
For example, the reaction system in 
Example~\ref{example1} is of order 2. However, if we use  
$\vvec v=(0,1,1,1)^\top$, then the reaction system in 
Example~\ref{example1} is of $\vvec v$-order 1. In general, 
if $\vvec v=(1,1,\dots,1)^\top$, the order and 
the $\vvec v-$order of a reaction are the same.

\section{Inference and identifiability of stochastic reaction systems}
\label{sec3}

Main results of this section are stated as
Theorems~\ref{thm:theoretical network inference}, 
\ref{thm:identify order N} and \ref{thm:non unique on a hyper}.

\subsection{Network inference using the transition rates}
\label{subsec:main}

Our goal is to construct a reaction system $(\Rx,\Ki)$ for given transition 
rates of a CTMC. First, we show that the knowledge of transition rates on
a sufficiently large part of the state space uniquely determines the 
underlying reaction system. 

\begin{lemma}\label{lem:same order systems are same}
Let $(\Rx,\Ki)$ and $(\overline{\Rx},\overline{\Ki})$ be two reaction systems 
of order $N_1$ and $N_2$, respectively. Suppose that there exists
$N \ge \max\{N_1,N_2\}$ such that the two mass-action 
stochastic models associated with $(\Rx,\Ki)$ and $(\overline{\Rx},\overline{\Ki})$ 
have the same transition rates on~$\mathbb S_{N}$. 
Then $(\Rx,\Ki)=(\overline{\Rx},\overline{\Ki})$.
\end{lemma}
\begin{proof}
Let ${\mathbf X}(t)$ and $\overline{\mathbf X}(t)$ be the CTMCs obtained 
by using stochastic mass-action description of $(\Rx,\Ki)$ and 
$(\overline{\Rx},\overline{\Ki})$, respectively.
We denote by $\lambda_{\vvec y\to \vvec y'}$ and $\overline \lambda_{\vvec y\to \vvec y'}$ 
the transition rates of reactions $\vvec y\to \vvec y'$ associated with 
$(\Rx,\Ki)$ and $(\overline{\Rx},\overline{\Ki})$, respectively. 
We denote states in 
$\mathbb S_N$ by (\ref{elemSN}). To prove the lemma by contradiction, we suppose 
that $(\Rx,\Ki) \neq (\overline{\Rx},\overline{\Ki})$. Since the order of each reaction 
$\vvec y\to \vvec y'$ in $\Rx \cup \overline{\Rx}$ is less than or equal to $N$, 
it can be represented as $\vvec x^k \to \vvec x^k+{\mathbf z}$ for some 
transition vector ${\mathbf z}$. Since 
$(\Rx,\Ki)\neq (\overline{\Rx},\overline{\Ki})$, there 
exists a transition vector ${\mathbf z}$ such that reaction 
$\vvec x^j\to \vvec x^j+{\mathbf z}$ is the first reaction (in the lexicographical 
ordering) which is formulated differently in reaction systems $(\Rx,\Ki)$ 
and $(\overline{\Rx},\overline{\Ki})$. In other words, we have
$\vvec x^i\to \vvec x^i+{\mathbf z} \in \Rx \cap \overline{\Rx}$ 
and $\lambda_{\vvec x^i \to \vvec x^i+{\mathbf z}}
\equiv
\overline \lambda_{\vvec x^i\to \vvec x^i+{\mathbf z}}$ for each $i<j$.
Then at $\vvec x^j \in \mathbb S_N$, the transition rate for ${\mathbf z}$ 
of the two systems are different, which is a contradiction to the assumption 
that both stochastic systems share the same transition rates on $\mathbb S_N$.
\end{proof}

\noindent Our main result is formulated as 
Theorem~\ref{thm:theoretical network inference} below, but before we state this 
theorem, we begin with a simple example illustrated in Figure~\ref{fig:example0}. 

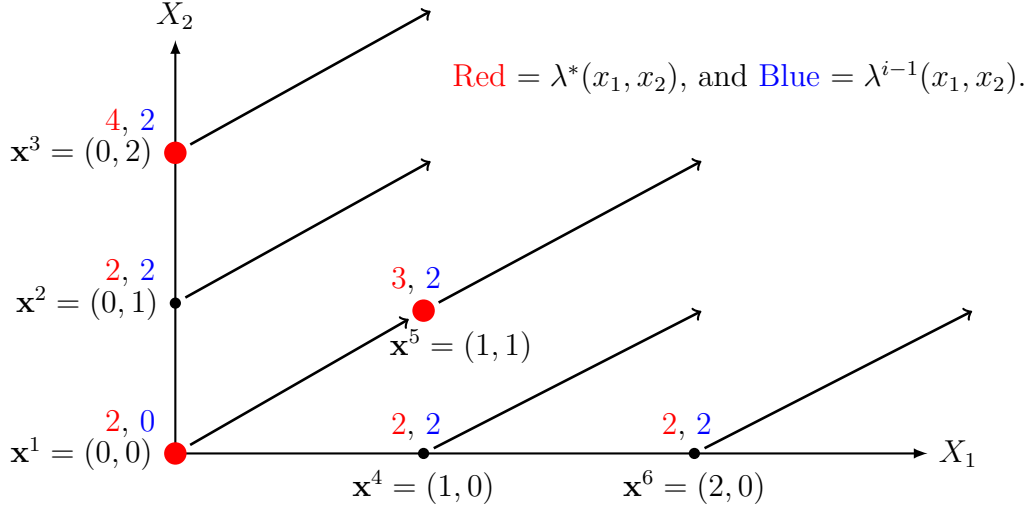
\begin{figure}[t]
\begin{tikzpicture}
\draw[thick,-latex] (-2,0) -- (8,0)node[right]{$X_1$};
\draw[thick,-latex] (-2,0) -- (-2,5.5)node[above]{$X_2$};
\node[fill,circle,red,inner sep=3pt,label=left:{${\mathbf x}^1=(0,0)$}] (0) 
at (-2,0) {};
\node at (-2.6,0.4) {\textcolor{red}{$2$}, \textcolor{blue}{$0$}};
\node[fill,circle,inner sep=1.5pt,label=left:{${\mathbf x}^2=(0,1)$}] (1)
at (-2,2) {};
\node at (-2.6,2.4) {\textcolor{red}{$2$}, \textcolor{blue}{$2$}};
\node[fill,circle,red,inner sep=3pt,label=left:{${\mathbf x}^3=(0,2)$}] (2) 
at (-2,4) {};
\node at (-2.6,4.4) {\textcolor{red}{$4$}, \textcolor{blue}{$2$}};
\node[fill,circle,inner sep=1.5pt,label=below:{${\mathbf x}^4=(1,0)$}] (3)
at (1.3,0) {};
\node at (1.2,0.35) {\textcolor{red}{$2$}, \textcolor{blue}{$2$}};   
\node[fill,circle,inner sep=1.5pt,label=below:{${\mathbf x}^6=(2,0)$}] (5)
at (4.9,0) {};
\node at (4.8,0.35) {\textcolor{red}{$2$}, \textcolor{blue}{$2$}};
\node[fill,circle,red,inner sep=3pt,
label=below right:{$\!\!\!\!\!\!\!\!\!\!{\mathbf x}^5=(1,1)$}] (4) at
(1.3,1.9) {};
\node at (1.2,2.3) {\textcolor{red}{$3$}, \textcolor{blue}{$2$}};
\node at (5.5,5) {\textcolor{red}{Red} 
$=\lambda^*(x_1,x_2)$, 
and \textcolor{blue}{Blue}  
$=\lambda^{i-1}(x_1,x_2)$.};
\node (17) at (5.2,2) { };
\node (18) at (8.8,2) { };
\node (19) at (1.6,4) { };
\node (20) at (5.2,4) { };
\node (21) at (1.6,6) { };
\path[shorten >=2pt,->,shorten <=2pt]
(0) edge[line width=1pt] node {} (4)
(1) edge[line width=1pt] node {} (19)
(2) edge[line width=1pt] node {} (21)
(3) edge[line width=1pt] node {} (17)
(4) edge[line width=1pt] node {} (20)
(5) edge[line width=1pt] node {} (18);
\end{tikzpicture}
\caption{{\it The procedure of inferring the underlying reaction system in
Example~$\ref{ex:main ex}$. 
The red value at each state is the given transition rate associated 
with the transition vector~${\mathbf z}=(1,1)^\top$, indicated by black arrows. 
The blue is the value of $\lambda^{i-1}$ updated at the previous state. Red 
dots indicate the states where $\lambda^*(x_1,x_2)-\lambda^{i-1}(x_1,x_2)>0$.}}
\label{fig:example0}        
\end{figure}

\begin{example}
\label{ex:main ex}
Consider $d=2$ and assume that the CTMC has a single transition 
vector ${\mathbf z}=(1,1)^T$. Suppose that we are given data on
transition rates $\lambda^*({\mathbf x})$ of a CTMC defined 
on $\ZZ^2_{\ge 0}$ as the red numbers indicated in
Figure~$\ref{fig:example0}$. To construct the reaction network,
we use $\mathbb S_2=\{\vvec x^1,\vvec x^2,\dots, \vvec x^6\}$ defined 
by $(\ref{defSN})$ and $(\ref{elemSN})$, i.e.
$$
\vvec x^1 = (0,0),
\quad
\vvec x^2 = (0,1),
\quad 
\vvec x^3 = (0,2),
\quad
\vvec x^4 = (1,0),
\quad 
\vvec x^5 = (1,1),
\quad
\vvec x^6 = (2,0).
$$
Let $\lambda^0\equiv 0$, $\Rx^0=\emptyset$ and $\Ki^0=\emptyset$. We iteratively 
calculate $\lambda^i$, $\Rx^i$ and $\Ki^i$ using given information at $\vvec x^i$, 
for $i=1,2,\dots,6.$ The outcome of this procedure is the transition rate 
function $\lambda= \lambda^6$, a set of reactions 
$\Rx=\Rx^6$, and a kinetic set $\Ki=\Ki^6$ such that 
\begin{align}
\sum_{{\mathbf y} \to {\mathbf y}' \in \Rx}
\lambda_{{\mathbf y} \to {\mathbf y}}(\vvec x)
=\lambda(\vvec x)=\lambda^*(\vvec x) \quad 
\text{for each $\vvec x\in \mathbb S_2$}.    
\end{align}
Since $\lambda^*(\vvec x^1)-\lambda^0(\vvec x^1) =2>0$ at $\vvec x^1=(0,0)$, 
the reaction $\emptyset \to X_1+X_2$ must be included in $\Rx$ with the reaction 
intensity $\lambda_{\emptyset\to X_1+X_2}({\mathbf x})=2$. So we let 
$$
\lambda^1({\mathbf x}) \equiv 2,
\qquad
\Rx^1=\left\{\emptyset \to X_1+X_2\right\}
\qquad
\mbox{and}
\qquad
\Ki^1=\left\{\lambda_{\emptyset\to X_1+X_2}({\mathbf x})=2\right\}.
$$
At the next state $\vvec x^2=(0,1)$, we have 
$\lambda^*(\vvec x^2)-\lambda^1(\vvec x^2)=0$, hence no additional reaction needs 
to be included in $\Rx$. Hence we put $\lambda^2 = \lambda^1,$
$\Rx^2 = \Rx^1$ and $\Ki^2=\Ki^1$. Since 
$\lambda^*(\vvec x^3)-\lambda^2(\vvec x^3) =2>0$ at $\vvec x^3=(0,2)$, 
the reaction $2 X_2 \to X_1 + 3 X_2$ must be included in $\Rx$ with the reaction 
intensity $\lambda_{2 X_2 \to X_1 + 3 X_2}({\mathbf x})= x_2 (x_2 - 1 )$. So we let 
$$
\lambda^3({\mathbf x}) = 2 + x_2 (x_2 - 1 ),
\qquad
\Rx^3=\left\{\emptyset \to X_1+X_2, \; 2 X_2 \to X_1 + 3 X_2 \right\}
$$
and
$
\Ki^3=\left\{\lambda_{\emptyset\to X_1+X_2}({\mathbf x})=2, \;
\lambda_{2X_2 \to X_1 + 3 X_2}({\mathbf x})=  x_2 (x_2 - 1 )\right\}.
$
We iterate this procedure until the last state $\vvec x^6=(2,0) \in \mathbb S_{2}$ as shown in Figure~$\ref{fig:example0}$. Then the outcome $(\Rx,\Ki)$ is the following 
reaction system 
\begin{align*}
\emptyset
\,\mathop{\longrightarrow}^{2}\,
X_1+X_2, \qquad 
2X_2 \,
\mathop{\longrightarrow}^1 \,
X_1+3X_2, \qquad 
X_1+X_2
\mathop{\longrightarrow}^1 
2X_1+2X_2, 
\end{align*} and the transition rate in the direction ${\mathbf z}=(1,1)^\top$ is  
\begin{eqnarray*}
\lambda({\mathbf x})
&=&
\lambda_{\emptyset\to X_1+X_2}({\mathbf x})
+
\lambda_{2X_2\to X_1+3X_2}({\mathbf x})
+
\lambda_{X_1+X_2\to 2X_1+2X_2}({\mathbf x})
\\
&=&
2+x_2(x_2-1)+x_1x_2.
\end{eqnarray*}
\end{example}

\noindent
We have observed in Example~\ref{ex:main ex} and 
Lemma~\ref{lem:same order systems are same} that a mass-action system of order 
$N=2$ can be characterized with the transition rates on $\mathbb S_N$. 
Next, we generalize this observation with a simple algorithm. 
Using the lexicographical order (\ref{elemSN}) of $\mathbb S_N$, 
for a given transition vector $\vvec z$ and the associated transition rate
$\lambda^*_{\vvec z}$ we iteratively define $\lambda^0_{\vvec z}\equiv 0$ and 
\begin{equation}
\label{eq:condition for the algorithm}
\lambda^i_\vvec z(\vvec x)
=
\lambda^{i-1}_\vvec z(\vvec x)+c_\vvec z^{i} \, {\vvec x}^{(\vvec x^i)}, 
\quad \text{for $i=1,2,\dots,n$, where} 
\quad
c_\vvec z^{i}
=
\dfrac{\lambda^*_{\vvec z}(\vvec x^i)
-
\lambda_\vvec z^{i-1}(\vvec x^i)}{{\vvec x^i}^{(\vvec x^i)}}.
\end{equation}
Note that 
$
\lambda^n_\vvec z (\vvec x)
=
\sum_{i\le n} c^i_\vvec z \, \vvec x^{(\vvec x^i)}
$, 
and the term $c^i_\vvec z \, \vvec x^{(\vvec x^i)}$ can 
be associated with the mass-action intensity of a reaction 
$\vvec x^i \to \vvec x^i+\vvec z$ as long as $c^i_\vvec z \ge 0$. 
Hence if $c^i_\vvec z \ge 0$ for each $i$, we can find 
a mass-action system that has the same transition rates 
as $\lambda^*_\vvec z$.

\begin{theorem}
\label{thm:theoretical network inference}
Let ${\mathbf X}(t)$ be a {\rm CTMC} defined on the state 
space $\mathbb Z^d_{\ge 0}$ with 
the transition rate $\lambda^*_{\mathbf z} : \mathbb Z^d_{\ge 0} \to [0,\infty)$ 
for each transition vector ${\mathbf z} \in {\cal Z} \subset \mathbb Z^d$, where
$|{\cal Z}| < \infty$. Suppose that the constant $c_\vvec z^{i}$ 
in $(\ref{eq:condition for the algorithm})$ is nonnegative
for each $\vvec z\in \cal Z$ and  $\vvec x^i \in \mathbb S_N$, 
for $i=1,2,\dots,n$, where we use notation~$(\ref{elemSN})$. 
Then for each integer $N>0$, there exists unique 
mass-action reaction system  $(\Rx,\Ki)$ such that \vskip 1mm
{\rm (i)} the order of the reaction system $(\Rx,\Ki)$ is less than or 
equal to $N$, and \vskip 1mm
{\rm (ii)} for each transition vector ${\mathbf z} \in {\cal Z}$, 
if $\lambda^*_{\mathbf z}(\vvec x')>0$ for some $\vvec x'\in \mathbb S_N$, 
then 
\begin{equation*}
\lambda^*_{\mathbf z}(\vvec x)
=\sum_{\substack{\vvec y\to \vvec y'\in \Rx \\
\vvec y'-\vvec y \, = \, {\mathbf z}}}\lambda_{\vvec y\to \vvec y'}(\vvec x) 
\quad \text{for all $\vvec x \in \mathbb S_N$},    
\end{equation*}
\hskip 5mm where $\lambda_{\vvec y \to \vvec y'}$ is the reaction intensity 
of $\vvec y \to \vvec y'\in \Rx$. 
\end{theorem}

\begin{proof} The uniqueness of $(\Rx,\Ki)$ follows from 
Lemma \ref{lem:same order systems are same}. To prove existence, we 
denote states in $\mathbb S_N$ by (\ref{elemSN}).
We fix ${\mathbf z} \in {\cal Z}$, and let $\lambda^*_{\mathbf z}$ be the associated 
transition rate function of $\vvec X(t)$. Then let 
\begin{align}
\begin{split}\label{eq: R and K}
&\Rx^\vvec z
=
\left\{
\vvec x^i \to \vvec x^i+\vvec z \;\;\big|\; \mbox{ where } i 
\mbox{ satisfies } c_\vvec z^{i} > 0
\right\}, \text{ and}\\
&\Ki^\vvec z
=
\left\{\lambda_{\vvec x^i\to \vvec x^i+\vvec z}(\vvec x) =
c_\vvec z^{i} \, \vvec x^{(\vvec x^i)} \;\Big|\;\, 
\vvec x^i\to \vvec x^i+\vvec z \in \Rx^\vvec z \right\}.
\end{split}
\end{align}
Then we prove that $\lambda^*_{\vvec z}(\vvec x^k)=\lambda_{\vvec z}^n(\vvec x^k)$ 
for each $\vvec x^k\in \mathbb S_N$, where $\lambda_{\vvec z}^n(\vvec x^k)$ is given
by (\ref{eq:condition for the algorithm}) and $n$ is given by (\ref{elemSN}).
Note that for any $k< j$, there is an $i$ such that 
$x^k_i\le  x^j_i$ so that ${\vvec x^k}^{(\vvec x^j)}=0$. Hence
\begin{align}
\label{eq:drop k prec terms}
\lambda_{\vvec z}^n(\vvec x^k)
=
\lambda_{\vvec z}^k(\vvec x^k)
+
\sum_{j=k+1}^n c_\vvec z^{j} \, {\vvec x^k}^{(\vvec x^j)}
=\lambda_{\vvec z}^k(\vvec x^k).
\end{align}
Therefore, for each $k$
$$
\lambda^*_{\vvec z}(\vvec x^{k})
= 
c_\vvec z^{k} \, {\vvec x^{k}}^{(\vvec x^{k})}
+
\lambda_{\vvec z}^{k-1}(\vvec x^{k})
=
\lambda_{\vvec z}^{k}(\vvec x^{k})
=
\lambda_{\vvec z}^{n}(\vvec x^{k}),
$$
where the last equality follows by \eqref{eq:drop k prec terms}. Repeating 
construction \eqref{eq: R and K} for each transition vector 
${\mathbf z} \in {\mathcal Z}$, we put
$$
\Rx = \bigcup_{{\mathbf z} \in {\cal Z}} \Rx^{\mathbf z},
\qquad
\mbox{and}
\qquad
\Ki = \bigcup_{{\mathbf z} \in {\cal Z}} \Ki^{\mathbf z},
$$
By the construction, for each transition vector ${\mathbf z} \in {\mathcal Z}$, 
we have 
$$
\lambda^{*}_{\mathbf z}(\vvec x)
= 
\lambda^{n}_{\mathbf z}(\vvec x)
=
\sum_{\substack{{\mathbf y}\to {\mathbf y}'\in \Rx \\ 
{\mathbf y}\to {\mathbf y}' \, = \, {\mathbf z}}}
\lambda_{{\mathbf y}\to {\mathbf y}'}(\vvec x),
$$
for each $\vvec x \in \mathbb S_N$, where $\lambda_{{\mathbf y}\to {\mathbf y}'}$ 
is the intensity of a reaction ${\mathbf y}\to {\mathbf y}'$ in $(\Rx,\Ki)$. 
The order of $(\Rx,\Ki)$ is less than or equal to $N$ since the order 
of each reaction in $\Rx$ is less than or equal to $N$.
\end{proof}

\begin{remark}
The advantage of Theorem~$\ref{thm:theoretical network inference}$ is that we do 
not require any algebraic structure on $\mathbb S_N$. Since the mass-action 
intensity of a reaction is a polynomial, transition rates on an arbitrary set 
$A$ can be used to infer the underlying reaction network and parameters by 
using a canonical polynomial fitting approach. To do that, however, certain 
algebraic structure on $A$ is required. More details about network inference 
with polynomial fitting are provided in Section~$\ref{sec:polynomial}$.
\end{remark}

\begin{remark}
If the transition rates $\lambda^*_{\vvec z}$ of a given CTMC $\vvec X(t)$ 
are given by an order $N$ mass-action system, then $c_{\vvec z}^i \ge 0$ 
for each transition vector ${\vvec z}$ and $i=1,2,\dots,n,$ and we can 
uncover the underlying reaction network uniquely by the algorithm
illustrated in Figure~$\ref{fig:example0}$.
\end{remark}

\subsection{Identifiability of continuous time Markov chains}

For a CTMC associated with a given reaction system, one of the main 
questions is identifiability of the underlying reaction system by using 
the information on the CTMC. We formalize this idea more rigorously.

\begin{definition}
\label{def:identify}
For a CTMC ${\mathbf X}(t)$ with the state space $\mathbb S$, 
the CTMC ${\mathbf X}(t)$ is \emph{identifiable} if there is 
a unique reaction system $(\Rx,\Ki)$ such that 
\begin{enumerate}
\item each ${\mathbf y}\to {\mathbf y}' \in \Rx$ is charged in at least one 
state $\vvec x \in \mathbb S$,
\item the state space of the CTMC associated with $(\Rx,\Ki)$ 
contains $\mathbb S$, and
\item the associated mass-action CTMC with $(\Rx,\Ki)$ admits the same transition 
rates on $\mathbb S$ as ${\mathbf X}(t)$ admits. 
\end{enumerate}
Otherwise, ${\mathbf X}(t)$  is not identifiable with a reaction system.
\end{definition}

\noindent
For a CTMC ${\mathbf X}(t)$ associated with an order $N$ reaction system, 
the uniqueness of Theorem~\ref{thm:theoretical network inference} implies 
that ${\mathbf X}(t)$ is identifiable as long as enough information on the 
transition rates of ${\mathbf X}(t)$ is ensured. We begin with a lemma 
for identifiability of reaction systems.

\begin{lemma}
\label{lem:relation bewteen order N1 and N2}
Let ${\mathbf X}_1(t)$ and ${\mathbf X}_2(t)$ be two $d$-dimensional 
{\rm CTMC}s associated with mass-action systems $(\Rx_1,\Ki_1)$ and $(\Rx_2,\Ki_2)$ 
of order $N_1$ and $N_2$, respectively. Suppose that $N_1 > N_2.$
Suppose further that ${\mathbf X}_1(t)$ and ${\mathbf X}_2(t)$
have the same transition rates at each state $\vvec x \in \mathbb S_{N_2}.$ 
Then $(\Rx_2,\Ki_2) \subset (\Rx_1,\Ki_1)$.
\end{lemma}
\begin{proof}
We apply Theorem~\ref{thm:theoretical network inference} to the transition rates 
of ${\mathbf X}_1(t)$ on $\mathbb S_{N_2}$ to identify a unique order $N'$ 
reaction system $(\Rx', \Ki')$ such that $N'\le N_2$ and the associated 
CTMC under mass-action kinetics has the same transition rates on~$\mathbb S_{N_2}$ 
as the transition rates of ${\mathbf X}_1(t)$. Then by the construction in the 
proof of Theorem~\ref{thm:theoretical network inference}, reaction
$\vvec y\to \vvec y' \in \Rx_2$ if and only if 
$\vvec y\to \vvec y'\in \Rx_1$ is of order $K$ for some $K\le N_2$.
That is, $\Rx'$ only contains a reaction in $\Rx_1$ whose order is less than 
or equal to $N_2$. Furthermore the reaction intensity of each reaction 
$\vvec y\to \vvec y' \in \Rx'$ is equal to the reaction intensity 
of $\vvec y\to \vvec y'\in \Rx_1$. Therefore $(\Rx',\Ki') \subset (\Rx_1,\Ki_1)$
Note also that since ${\mathbf X}_1(t)$ and ${\mathbf X}_2(t)$ have the same 
transition rates on $\mathbb S_{N_2}$, by uniqueness shown in 
Lemma~\ref{lem:same order systems are same}, we have
$(\Rx',\Ki')=(\Rx_2,\Ki_2)$ because the order of both 
reaction systems are less than or equal to $N_2$, and the associated 
CTMC's have the same transition rates on $\mathbb S_{N_2}$. 
\end{proof}

\noindent
Lemma~\ref{lem:relation bewteen order N1 and N2} ensures that if two reaction systems 
have the same transition rates, then the one with lower order is a subsystem of the 
other. Using this fact, we obtain identifiability of a reaction system.

\begin{theorem}\label{thm:identify order N}
Let $\vvec X(t)$ be a {\rm CTMC} associated with an order $N$ reaction 
system $(\Rx,\Ki)$ with the state space $\mathbb S$. 
If $\mathbb S_{N} \subseteq \mathbb S$, then $\vvec X(t)$ is identifiable.
\end{theorem}
\begin{proof}
First of all, suppose that there exists a reaction system 
$(\overline \Rx,\overline \Ki)$ of order $\overline N$ where
$\overline{N}<N$ such that the associated mass-action system satisfies 
the conditions (1)-(3) in Definition~\ref{def:identify}. 
Then Lemma~\ref{lem:relation bewteen order N1 and N2} implies that 
$(\overline \Rx,\overline \Ki) \subset (\Rx,\Ki)$. Since $\overline N<N$, 
there exists a reaction $\tilde{\vvec y} \to \tilde{\vvec y}'$ of order $N$ 
that belongs to $\Rx\setminus \overline{\Rx}$. 
Let ${\mathbf z}=\tilde{\vvec y}'-\tilde{\vvec y}$. Then at 
state $\tilde{\vvec y} \in \mathbb S_N$, 
$$
\sum_{\substack{\vvec y\to \vvec y' \in \Rx\\ \vvec y'-\vvec y={\mathbf z}}}
\lambda_{\vvec y\to \vvec y'}(\tilde{\vvec y})
\; -
\sum_{\substack{\vvec y\to \vvec y' \in \overline \Rx\\ \vvec y'-\vvec y={\mathbf z}}}
\overline \lambda_{\vvec y\to \vvec y'}(\tilde{\vvec y}) 
\ge 
\lambda_{\tilde{\vvec y}\to\tilde{\vvec y}}(\tilde{\vvec y})
>
0,
$$
where $\lambda_{\vvec y \to \vvec y}$ and $\overline \lambda_{\vvec y \to \vvec y}$ 
are the reaction intensity associated with a reaction $\vvec y \to \vvec y'$ of
$(\Rx,\Ki)$ and $(\overline \Rx,\overline \Ki)$, respectively. Therefore 
it contradicts to the fact that $(\overline \Rx,\overline \Ki)$ has the same 
transition rates on each state $\vvec x \in \mathbb S_N$ as $\vvec X(t)$. 
For the same reason, there does not exists a reaction system, which has higher 
order than $N$, satisfies the conditions (1)-(3) in Definition~\ref{def:identify}. 

In conclusion, the only reaction network satisfying the conditions (1)-(3) 
in Definition~\ref{def:identify} is $(\Rx,\Ki)$ because uniqueness among 
reaction systems of order $N$ is guaranteed by 
Lemma~\ref{lem:same order systems are same} and $(\Rx,\Ki)$ satisfies 
the condition (1)-(3) in Definition~\ref{def:identify}.
\end{proof}

\noindent
In practical situations, it is often that an associated mass-action 
CTMC $\vvec X(t)$ is given, but the underlying reaction system $(\Rx,\Ki)$ 
is unknown. However, it is reasonable to assume that the order of 
$(\Rx,\Ki)$ does not exceed a relatively small number $\overline N$ 
for general biochemical system (for example, many biochemical systems 
are at most bimolecular, hence we could set $\overline N=2$). Under 
this assumption, $\vvec X(t)$ is identifiable as long as enough 
information about the transition rates is given. The case of the unknown
order is a consequence of Theorems~\ref{thm:theoretical network inference} and \ref{thm:identify order N} and is formulated as the following corrollary.

\begin{corollary}\label{cor:unknown order}
Let a CTMC $\vvec X(t)$ be a mass-action stochastic system associated 
with an unknown order $N$ reaction system $(\Rx,\Ki)$ with the state 
space $\mathbb S$. Suppose that $N \le \overline{N}$ for some positive 
integer $\overline{N}$. Suppose further that $\mathbb S_{\bar N} \subseteq \mathbb S$.
Then $\vvec X(t)$ is identifiable. Moreover, by using the transition rates 
of $\vvec X(t)$, the true network $(\Rx,\Ki)$ can be explicitly inferred.
\end{corollary}

\subsection{Identifiability of reaction systems with 
conservation laws}

If the transition rate of a Markov process is given over a proper
subset $\mathbb A \subset \mathbb S_N$ for given $N>0$, then 
two distinct reaction systems of order $N$ may be constructed 
having the same transition rates over $\mathbb A$. Since 
$\mathbb A$ is the proper subset 
of $\mathbb S_N$, we have 
$
\mathbb A \subset \{\vvec x^1,\vvec x^2,\dots,\vvec x^m\} 
\subset \mathbb S_N
$ where 
$m<n$. 
Given the transition rates on 
$\{\vvec x^1,\vvec x^2,\dots,\vvec x^m\}$
and considering $\vvec x^\ell \succ \vvec x^m$ such that 
$\vvec x^\ell \in \mathbb S_N$, the mass-action reaction 
intensity associated with a reaction 
$\vvec x_\ell \to \vvec x_\ell +\vvec z$ is zero at each 
state in $\{\vvec x^1,\vvec x^2,\dots,\vvec x^m\}$. Hence 
by adding or removing $\vvec x_\ell\to \vvec x_\ell +\eta$, 
we obtain different reaction systems that have the 
same transition rates on $\mathbb A$.

Next, we consider other situations where the underlying reaction
system of a CTMC is not uniquely determined. Suppose a given CTMC
$\vvec X(t)$ associated with a stochastic reaction
network of order $N$ admits a conservation
law, i.e. there exists $\vvec v\in\ZZ^d_{\ge 0}$ such that
$\vvec v \cdot \vvec X(t)=\vvec v\cdot \vvec X(0)$ for any time
$t \ge 0$. In this section, we simplify our discussion 
by considering that the vector $\vvec v$ has all non-zero 
components, that is $\vvec v\in\ZZ^d_{> 0}$. Then
the state space of $\vvec X(t)$ is confined to a finite hyperplane 
$\mathbb S_{\vvec v, N}$ of $\mathbb Z^d_{> 0}$. 
In this case, one of the main questions is whether 
the information about the transition rates over a single hyperplane 
is sufficient to uniquely infer the underlying reaction system. 

In this section, we show how to construct a reaction network 
of order~$N$ with given transition rates over a single
hyperplane~$\mathbb S_{\vvec v, N}$, see the definition
(\ref{defSvN}). We further show that when a given reaction 
system $(\Rx,\Ki)$ is of order $N$, then the underlying reaction 
network is not uniquely identified with given transition rates 
on a single hyperplane $\mathbb S_{\vvec v, N'}$ such that $N<N'$. 

\begin{theorem}\label{thm:ex of CRN on conv}
Let $\vvec z\in \ZZ^d$, $\vvec v \in \ZZ^d_{> 0}$ and $N>0$.
Let $\lambda(\vvec x)$ be a given non-negative
function defined on $\mathbb S_{\vvec v,N}$  such that 
$\lambda(\vvec x)>0$ for at least 
one $\vvec x\in \mathbb S_{\vvec v,N}$.
Then there exists a mass-action reaction system 
$(\Rx^{\vvec z},\Ki^{\vvec z})$ of $\vvec v$-order $N$ such that 
the transition rates at each $\vvec x \in \mathbb S_{\vvec v,N}$ 
are equal to $\lambda(\vvec x)$. That is
\begin{equation}
\label{statthe}
\sum_{\substack{\vvec y\to \vvec y' \in \Rx^{\vvec z}}}
\lambda_{\vvec y\to \vvec y'}(\vvec x)=\lambda(\vvec x) 
\quad \text{for each} \quad 
\vvec x \in \mathbb S_{\vvec v, N}.
\end{equation}
\end{theorem}

\begin{proof}
The key idea of the proof is that (under the mass-action 
kinetics) every reaction of $\vvec v$-order $N$ is charged 
at a single state $\vvec x \in \mathbb S_{\vvec v, N}$ and turned 
off elsewhere in $\mathbb S_{\vvec v, N}$. So we will collect all 
reactions $\vvec x\to \vvec x+\vvec z$ for each 
$\vvec x\in\mathbb S_{\vvec v,N}$ as long as 
$\lambda(\vvec x)>0.$ 
We define 
\begin{align*}
&
\Rx^{\vvec z}
=
\left\{
\vvec x\to \vvec x+\vvec z 
\;\big|\; 
\lambda(\vvec x)>0, \; \vvec  x \in \mathbb S_{\vvec v,N} 
\right\} \text{, and} \\
&
\Ki^{\vvec z}
=
\left\{\lambda_{\vvec x\to \vvec x+\vvec z}(\vvec w)
=
\frac{\lambda(\vvec x)}{\vvec x^{(\vvec x)}} \, \vvec w^{(\vvec x)}
\text{ for any $\vvec w\in\ZZ^d_{\ge 0}$} 
\; \Big| \; 
\lambda(\vvec x)>0, \; \vvec x \in \mathbb S_{\vvec v,N} 
\right 
\}.
\end{align*}
Since $\vvec v \in \ZZ^d_{> 0}$, for any two distinct states 
$\vvec x$ and $\vvec x'$ in $\mathbb S_{\vvec v,N}$, there is an 
index $k$ such that $x_k > x'_k$. Therefore the reaction 
$\vvec x \to \vvec x +\vvec z \in \Rx^{\vvec x}$ is turned off 
at $\vvec x'$ if and only if $\vvec x\neq \vvec x'$. This implies 
that for any $\vvec x \in \mathbb{S}_{\vvec v,N}$ such that 
$\lambda(\vvec x) > 0$,
$$
\lambda(\vvec x)
=
\frac{\lambda(\vvec x) }{\vvec x^{(\vvec  x)}}
\, \vvec x^{(\vvec x)}
=
\lambda_{\vvec x \to \vvec x +\vvec z}(\vvec x)
=
\sum_{\vvec y\to \vvec y' \in \Rx^{\vvec z}}
\lambda_{\vvec y\to \vvec y'}(\vvec x).
$$
Equation (\ref{statthe}) is also valid for any 
$\vvec x\in \mathbb S_{\vvec v,N}$ satisfying 
$\lambda(\vvec x)=0$, because we have 
$(\vvec x\to \vvec x +\vvec z) \not \in \Rx^{\vvec z}$ and each 
$\vvec x'\to \vvec x'+\vvec z \in \Rx_{\vvec z}$ is turned off 
at $\vvec x$.
\end{proof}

\noindent
Theorem \ref{thm:ex of CRN on conv} implies that for 
a given CTMC defined on a hyperplane $\mathbb S_{\vvec v,N}$, 
we can construct a reaction network of $\vvec v$-order $N$ such 
that the associated mass-action CTMC admits the same transition 
rates on $\mathbb S_{\vvec v,N}$. By using this, we prove that 
a CTMC associated with a conservative reaction system of 
$\vvec v$-order $N$ is not identifiable if the transition data 
of the CTMC are only given on $\mathbb S_{\vvec v,N'}$ for 
some $N' > N$.

\begin{theorem}\label{thm:non unique on a hyper}
Let  $(\Rx,\Ki)$ be a mass-action reaction system that admits 
a conservation law with $\vvec v\in \ZZ^d_{> 0}$ such 
that $\vvec v \cdot (\vvec y'-\vvec y)=0$ for each 
$\vvec y\to \vvec y'\in \Rx$. Suppose that the $\vvec v$-order 
of $(\Rx,\Ki)$ is $N$. Let $\vvec X(t)$ be the CTMC associated 
with $(\Rx,\Ki)$ such that $\vvec v\cdot \vvec X(0)=N'$
and $N'> N$. Then the CTMC $\vvec X(t)$ is not identifiable.
\end{theorem}
\begin{proof}
Because of the conservation law, the state space of 
$\vvec X(t)$ is $\mathbb S_{\vvec v,N'}$, defined by (\ref{defSvN}),
because $\vvec v\cdot \vvec X(t)=\vvec v\cdot \vvec X(0)$ 
for any time $t \ge 0$. For a fixed transition vector 
$\vvec z$ in the set of transition vectors $\cal Z$ of 
$\vvec X(t)$, we denote by $\lambda_{\vvec z} (\vvec x)$ the 
transition rate of $\vvec X(t)$ at $\vvec x\in \mathbb S_{\vvec v,N'}$. 
Then for each $\vvec x\in \mathbb S_{\vvec v,N'}$, we have
$$
\lambda_{\vvec z}(\vvec x)
=
\sum_{\substack{\vvec y\to \vvec y' \in \Rx\\ 
\vvec y'-\vvec y\,=\,\vvec z}}
\lambda_{\vvec y\to \vvec y'}(\vvec x),
$$
where $\lambda_{\vvec y\to \vvec y'}$ is the intensity of 
reaction $(\vvec y\to \vvec y') \in \Rx$. 
Since $\lambda_{\vvec z}$ is the transition rate of 
a mass-action reaction system of $\vvec v$-order equal to $N$, 
there exists $\vvec x^*\in \mathbb S_{\vvec v, N}$ such that 
$(\vvec x^*\to \vvec x^*+\vvec z) \in \Rx$. 
Therefore for 
$
\vvec x'
=
\vvec x^*+((N'-N)/v_1,0,0,\dots,0)^\top\in \mathbb S_{\vvec v,N'}
$, we have
$
\lambda_{\vvec z}(\vvec x')\ge \lambda_{\vvec x^*\to \vvec x^* +\vvec z}(\vvec x') >0.
$
This means that there exist at least one 
$\vvec x' \in \mathbb S_{\vvec v,N'}$ 
such that $\lambda_{\vvec z}(\vvec x')>0$. Hence by using 
Theorem~\ref{thm:ex of CRN on conv} with $\lambda_{\vvec z}$ 
and $\vvec z$, we can construct a reaction system 
$(\Rx^{\vvec z},\Ki^{\vvec z})$ of $\vvec v$-order $N'$. 
Then we have
\begin{align*}
    \lambda_{\vvec z}(\vvec x)=\sum_{\vvec y\to \vvec y' \in \Rx^{\vvec z}}\overline \lambda_{\vvec y\to \vvec y'}(\vvec x),
\end{align*}
where $\overline \lambda_{\vvec y\to \vvec y'}$ is the intensity of
reaction $\vvec y\to \vvec y'$ in $(\Rx^{\vvec z},\Ki^{\vvec z})$.
Applying Theorem~\ref{thm:ex of CRN on conv} in the same way 
for all transition vectors $\vvec z \in \cal Z$, we define
$$
\overline \Rx
=
\bigcup_{\vvec z \in \mathcal Z}\Rx^{\vvec z}, 
\quad \text{and} \quad 
\overline \Ki=\bigcup_{\vvec z\in \mathcal Z}\Ki^{\vvec z}.
$$
Then we have 
$$
\lambda_{\vvec z}(\vvec x)
=
\sum_{\substack{\vvec y\to \vvec y' \in \Rx^{\vvec z}}}
\overline \lambda_{\vvec y\to \vvec y'}(\vvec x)
=
\sum_{\substack{\vvec y\to \vvec y' \in \overline \Rx \\ 
\vvec y'-\vvec  y \, = \, \vvec z}}
\overline \lambda_{\vvec y\to \vvec y'}(\vvec x),
$$
for each $\vvec x \in \mathbb S_{\vvec v,N'}$
and for each transition vector $\vvec z$ of $\vvec X(t)$.
This implies that the CTMC associated with 
$(\overline \Rx,\overline \Ki)$ has the same transition rates 
on $\mathbb S_{\vvec v,N'}$, which is the state space of 
$\vvec X(t)$. Since $(\overline\Rx,\overline \Ki)$ is of 
$\vvec v$-order $N'$, two reaction systems $(\Rx,\Ki)$ 
and $(\overline \Rx,\overline \Ki)$ are distinct. 
Hence $\vvec X(t)$ is not identifiable.
\end{proof}

\noindent
We illustrate Theorem~\ref{thm:non unique on a hyper} 
using the following example.

\begin{example}\label{ex:non unique1}
Let $\vvec X(t)$ be the CTMC associated with the 
mass-action reaction system
\begin{align}
\label{eq:example network1}
X_1
\xrightleftarrows{1}{1}{4 mm}
X_2 \, .
\end{align}
Note that this system admits a conservation law such that 
$\vvec v\cdot \vvec X 
= 
X_1(0)+ X_2(0)
$ where $\vvec v=(1,1)^\top \in \ZZ^2_{> 0}$.
With $\vvec X(0)=(2,0)$, the transition rates of $\vvec X(t)$ 
at its state space $\mathbb S_{\vvec v,2}$ are
\begin{equation}
\lambda_{(-1,1)}(2,0) = \lambda_{(1,-1)}(0,2) = 2, 
\qquad \mbox{and} \qquad
\lambda_{(-1,1)}(1,1) = \lambda_{(1,-1)}(1,1) = 1.
\label{extransitionrates}
\end{equation}
Note that the $\vvec v$-order of the 
reaction system~$(\ref{eq:example network1})$ is $1$.
By using Theorem~$\ref{thm:ex of CRN on conv}$, we construct 
the following reaction system of $\vvec v$-order $2$ with the 
the same transition rates~$(\ref{extransitionrates})$
on $\mathbb S_{\vvec v,2}$:
\begin{align}\label{eq:example network2}
2 \, X_1
\xrightleftarrows{1}{1}{4 mm}
X_1+X_2 
\xrightleftarrows{1}{1}{4 mm}
2 \, X_2 \, .
\end{align}
The CTMC associated with the reaction system~$(\ref{eq:example network2})$
admits the same transition rates on $\mathbb S_{\vvec v,2}$ as 
$\vvec X(t)$ does. However, these two reaction systems exhibit 
different dynamical behaviours if we consider them on a different
hyperplane $\mathbb S_{\vvec v,N}$ as we show
in Figure~$\ref{fig:example}$ for $N=4$. Considering the initial
condition $\vvec X(0)=(N,0)$,  the mean and variance of the number of
molecules of $X_1$ of the reaction system~$(\ref{eq:example network1})$
are given by~\cite{Erban:2020:SMR}
\begin{equation}
\langle X_1 \rangle = \frac{N}{2}
\,\big(
1 + \exp[-2t]
\big),
\qquad
\langle X_1^2 \rangle - \langle X_1 \rangle^2 
= 
\frac{N}{4}
\,\big(1 - \exp[-4 t] \big).
\label{meanvariance36}
\end{equation}
Using $N=4$, we plot~$(\ref{meanvariance36})$ as the red solid lines
in Figure~$\ref{fig:example}$, where we compare them with the results
calculated for the reaction system~$(\ref{eq:example network2})$
by averaging over $10^7$ realizations of the Gillespie 
stochastic simulation algorithm (SSA).

\begin{figure}[tb]
\leftline{(a) \hskip 7.9cm (b)}
\centerline{\hskip 8mm \includegraphics[height=6.2cm]{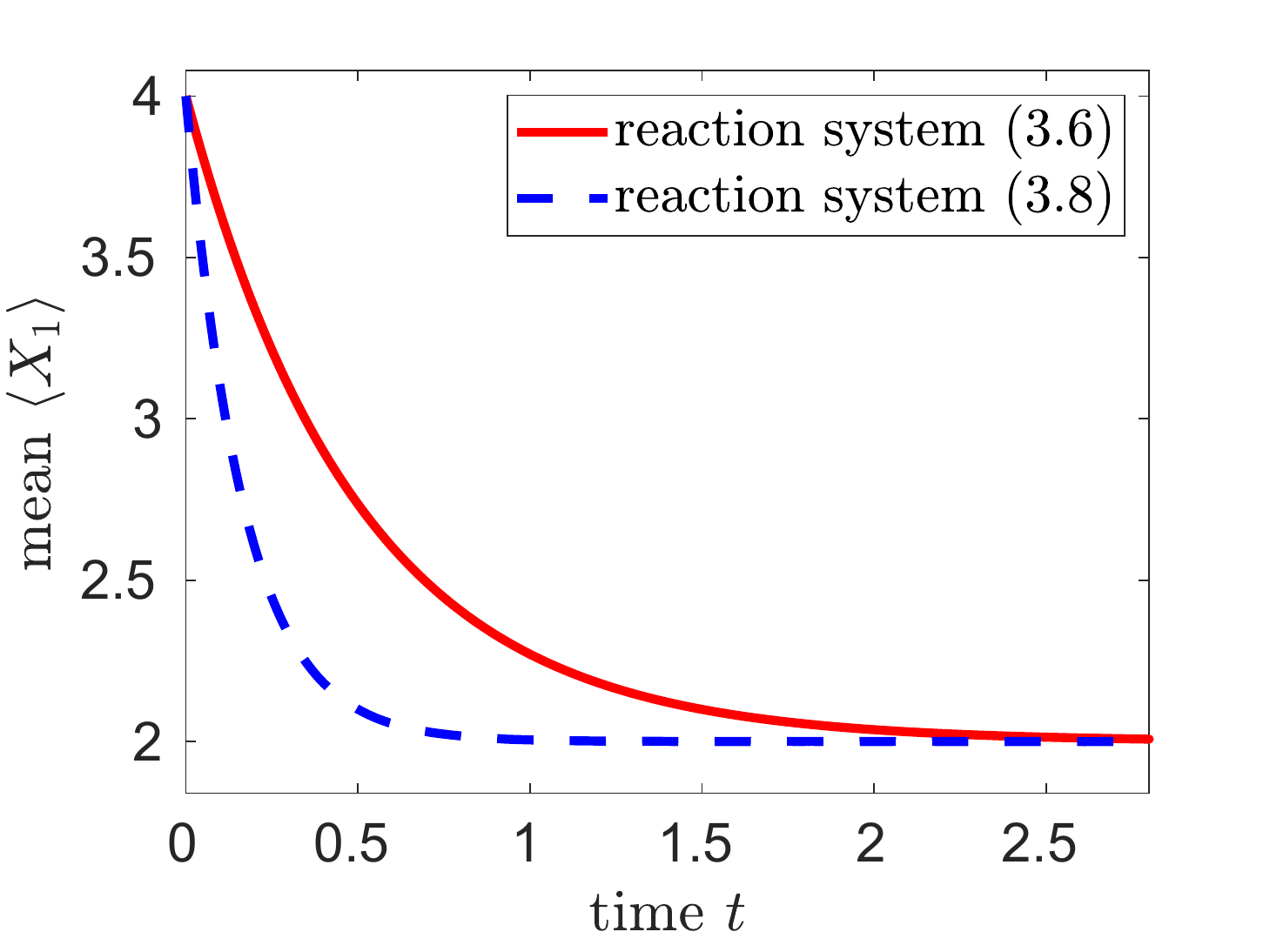}
$\;\;$\includegraphics[height=6.2cm]{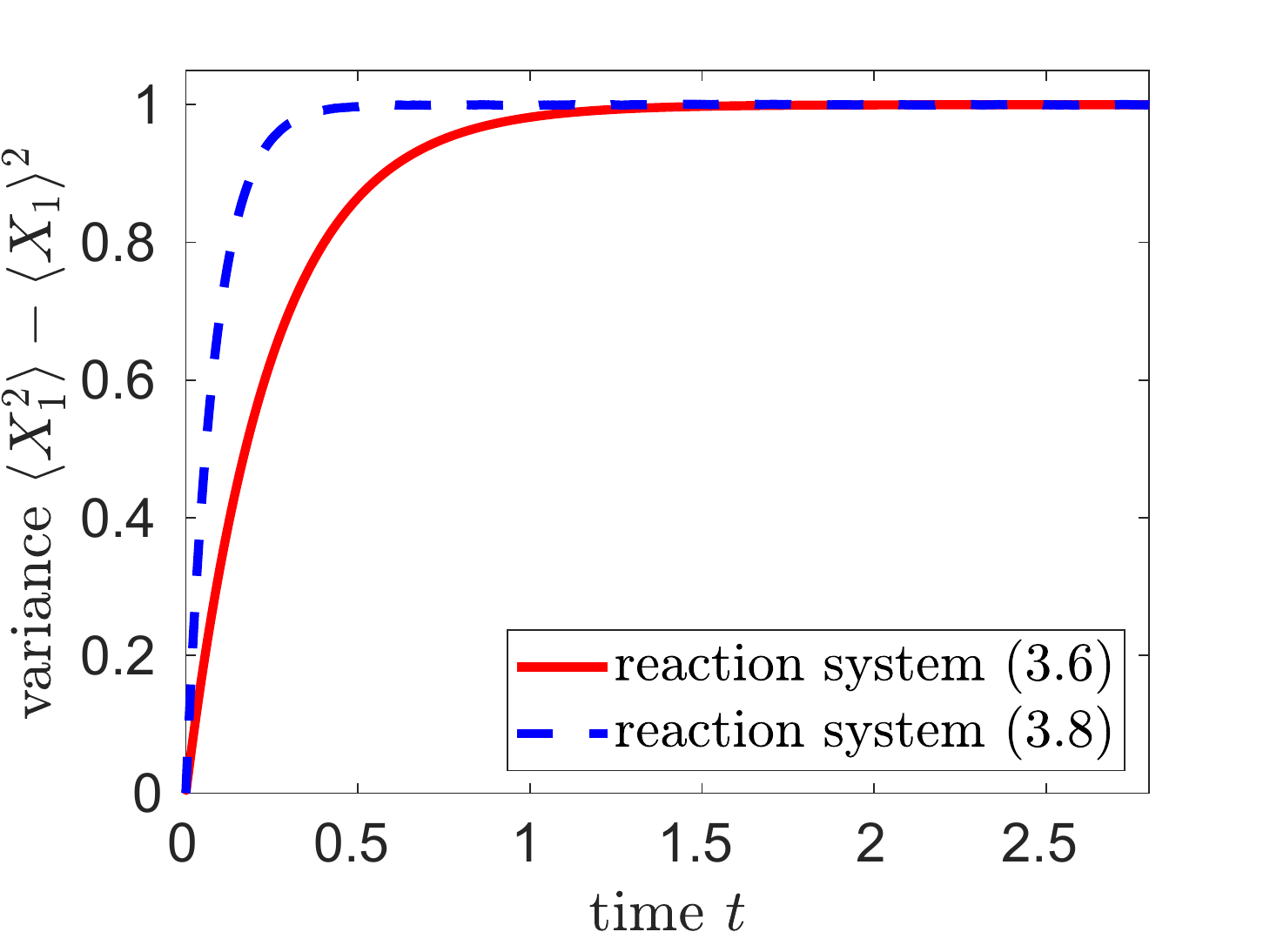}}
\caption{(a) {\it The mean number of molecules of the chemical species
$X_1$ for the chemical system~$(\ref{eq:example network1})$ 
(red solid line) compared with the result for the chemical 
system~$(\ref{eq:example network2})$ (blue dashed line).} \\
(b) {\it Time evolution of the variance of the number of 
molecules of the chemical species $X_1$. We use
the same initial condition $\vvec X(0)=(N,0)$, where $N=4$, 
for both systems. The results for the reaction 
network~$(\ref{eq:example network1})$ are calculated by
equation~$(\ref{meanvariance36})$, while the results for the
reaction network~$(\ref{eq:example network2})$ are estimated
as averages over $10^7$ realizations of the Gillespie SSA.} 
}
\label{fig:example}
\end{figure}
\end{example}

\begin{remark}
If we consider the same hyperplane, $\mathbb S_{\vvec v,2}$, as 
in Example~$\ref{ex:non unique1}$, we can also construct an identifiable
network if the conditions of Theorem~$\ref{thm:non unique on a hyper}$ are not satisfied. For example, replacing the reaction 
system~$(\ref{eq:example network1})$ with the reaction system
$$
2 \, X_1 
\xrightleftarrows{\kappa_1}{\kappa_2}{4 mm}
2 \, X_2 \, .
$$
and letting $X(0)=(2,0)$, the state space is 
$\{(2,0),(0,2)\} \subset \mathbb S_{\vvec v,2}$ with 
$\vvec v=(1,1)^\top$. Then the CTMC $\vvec X(t)$ is the only
reaction network of the $\vvec v$-order $2$ with the
same transition rates on $\mathbb S_{\vvec v,2}$, that is,
the CTMC $\vvec X(t)$ is identifiable.
\end{remark}

\section{Reaction networks for Markov processes with polynomial rates}
\label{sec:polynomial}
In Section~\ref{subsec:main}, we showed that if the transition rates 
of a CTMC are given at each state in $\mathbb S_N$ for some $N$, 
then we can uniquely identify an order $N$ stochastic reaction 
system that has the same transition rates on $\mathbb S_N$.
In this section, we explore the case where the transition rates 
of a CTMC are known on arbitrary states, which are not necessarily
belonging to $\mathbb S_N$. For a $d$-dimensional CTMC, we will use 
the transition rates at (compare with (\ref{defSN})
and (\ref{elemSN}))
$$
n
=
|\mathbb S_N|
=
\binom{N+d}{d}
$$ 
different states to uniquely identify an order $N$ stochastic 
reaction system that has the same transition rates at the given 
states. 

\begin{lemma}\label{lem:find polynomial}
Let $\vvec X(t)$ be a CTMC defined on $\ZZ^d_{\ge 0}$ with the 
finite set of transition vectors~$\cal Z$. Suppose for each transition 
vector $\vvec z \in \cal Z$, the transition rates of $\vvec  X(t)$ 
are given in finite set $A_{\vvec z} \subset \ZZ^d$. Then 
there exists a CTMC $\overline{\vvec X}(t)$ with polynomial transition 
rates such that for each $\vvec z\in \cal Z$ 
\begin{align}
\label{eq:find poly}
\lambda_{\vvec z}(\vvec x)
=
\overline \lambda_{\vvec z}(\vvec x) \quad 
\text{for each $\vvec x\in A_{\vvec z}$},
\end{align}
where $\lambda_{\vvec z}$ is the given transition rate of 
$\vvec X(t)$, and $\overline\lambda_{\vvec z}$ is a polynomial 
transition rate of~$\overline{\vvec X}(t)$. 
Moreover, assume that we have $|A_{\vvec z}|=n=|\mathbb S_N|$ 
for some positive integer $N$, and denote the elements of 
$A_{\vvec z}$ as ${\vvec a}^1$, ${\vvec a}^2,$ $\dots$, 
${\vvec a}^n$ and elements of $\mathbb S_N$ by $(\ref{elemSN})$.
Define matrix $M \in \ZZ^{n \times n}$ with entries 
$$
M_{ij}={{\vvec a}^i}^{({\vvec x}^j)}
\qquad \mbox{ for } \quad 
i=1,2,\dots, n, \; j=1,2,\dots,n.
$$ If matrix $M$ is invertible, then 
$\lambda_{\vvec z}$ is a unique degree $N$ polynomial.
\end{lemma}

\begin{proof}
We can find a polynomial $\overline \lambda_{\vvec z}$ such that
(\ref{eq:find poly}) is satisfied because the set
$A_{\vvec z}$ is finite for each transition vector 
$\vvec z\in \cal Z$ and $|\cal Z| < \infty.$
Suppose that matrix $M$ is invertible. Note that for each 
$\vvec z\in \cal Z$, we let $\vvec c \in \mathbb{R}^n$ 
such that 
\begin{equation}
\vvec c
=
M^{-1}\vvec b, 
\quad \text{where}
\quad 
b_i= \overline \lambda_{\vvec z}({\vvec a}^i) 
\quad \mbox{for each} \quad i=1,2,\dots,n.
\label{cMb}
\end{equation}
Then the degree $N$ polynomial  
$\overline  \lambda_{\vvec z} (\vvec x)$ 
is uniquely written as
\[ 
\overline  \lambda_{\vvec z}(\vvec x)
=
\sum_{j=1}^n
c_j \, \vvec x^{(\vvec x^j)}.
\]
\end{proof}

\noindent
For a given CTMC, our final goal of this section is to identify 
a unique mass-action stochastic system that has the same transition 
rates as the given CTMC admits. By applying 
Lemma~\ref{lem:find polynomial}, we can construct a CMTC whose 
transition rates are polynomials and have the same values as the 
given transition rates. However, not every CTMC with polynomial 
rates is associated with a mass-action reaction network. Negative
coefficients cause problems as it is the case of polynomial ODE 
models which cannot be written as chemical reaction
systems~\cite{Plesa:2016:CRS}. In the case of CTMC the situation 
is even more restrictive. To formulate the theorem
characterizing which CTMC with polynomial transition rates 
can be identified as a mass-action reaction system,
we denote by $D_i(\lambda)$ the minimum power of $x_i$ in the 
polynomial $\lambda(1,1,\dots, x_i,\dots,1)$, where
$\lambda : \mathbb{Z}^d \to \mathbb{R}$ is a polynomial.
For instance, if $\lambda( x_1,x_2)= x_1^3 x_2^2+ x_1$, 
then $D_1(\lambda)=1$ and $D_2(\lambda)=2$.

\begin{theorem}\label{thm:polynomial and reaction network}
Let $\vvec X(t)$ be a CTMC defined on $\mathbb{Z}^d_{\ge 0}$ with the set 
of transition vectors $\cal Z$. Suppose that each transition rate
$\lambda_{\vvec z}$ of $\vvec X(t)$ associated with $\vvec z \in \cal Z$ 
is a polynomial of degree $N$ such that
\begin{align}
\label{eq:poly decompose}
\lambda_\vvec z(\vvec x)
=
\sum_{j=1}^n
c_j \, \vvec x^{(\vvec x^j)}
\text{for some constants $c_j\ge 0$},
\end{align}
where $n=|\mathbb S_N|$ and elements of $\mathbb S_N$ are
denoted by $(\ref{elemSN})$. Suppose further that
\begin{align}\label{eq:polynomial condition}
|z_i| \le D_i(\lambda_\vvec z) \quad \text{if $z_i < 0$}.
\end{align} 
Then there exists a unique mass-action reaction system 
such that the associated mass-action stochastic model 
is equal to the CTMC $\vvec X(t)$. 
\end{theorem}
\begin{proof}
Let $\vvec z \in \cal Z$ be fixed. Then the associated transition 
rate $\lambda_{\vvec z}$ is given by~\eqref{eq:poly decompose}.  
Note that equation~\eqref{eq:polynomial condition} implies that 
$x^j_i + z_i \ge 0$ for every term $\vvec x^{(\vvec x^j)}$ 
in~\eqref{eq:poly decompose}. Therefore we define 
$
\Rx^{\vvec z}
=\{\vvec x^j \to \vvec x^j +\vvec z \, | \, c_j > 0
\}
$ 
and 
$
\Ki^{\vvec z}
=
\{\lambda_{\vvec  x^j \to \vvec x^j + \vvec z}(\vvec x) 
=
c_j \, \vvec x^{({\vvec x}^j)} 
\; | \; c_j > 0\}
$. 
Then 
$$
\lambda_\vvec z(\vvec x)
=
\sum_{\vvec x^j \to \vvec x^j + \vvec z \in \Rx_\vvec z} 
\lambda_{\vvec x^j \to \vvec x^j + \vvec z}(\vvec x). 
$$
Considering $\Rx^\vvec z$ and $\Ki^\vvec z$ obtained for each 
$\vvec z \in \cal Z$, we define
$\Rx=\bigcup_{\vvec z\in \cal Z}\Rx^\vvec z$ 
and $\Ki=\bigcup_{\vvec z \in \cal Z}\Ki^\vvec z$. 
The associated CTMC for $(\Rx,\Ki)$ has the same transition 
rates as $X$ has. Uniqueness follows since the 
decomposition~\eqref{eq:poly decompose} is unique.
\end{proof}

\noindent
Suppose a given CTMC satisfies the conditions in  
Lemma~\ref{lem:find polynomial} and that the transition 
rates of the CTMC satisfy the conditions~\eqref{eq:poly decompose} 
and~\eqref{eq:polynomial condition} in 
Theorem~\ref{thm:polynomial and reaction network}. Then we can 
infer a reaction network whose associated CTMC has the same 
transition vectors and the same transition rates at each state 
in $A_\vvec z$ for each transition vector $\vvec z$. 
We demonstrate this using the following example.

\begin{example}
Let $\vvec  X(t)$ be a CTMC defined on $\mathbb Z^2_{\ge 0}$. 
Suppose that it is known that $\vvec  X(t)$ admits three transition 
vectors $\vvec z^1=(1,0)^\top$,  $\vvec z^2=(-1,1)^\top$ 
and $\vvec z^3=(0,-1)^\top$. We are also given information 
on the transition rates of $\vvec X(t)$ as 
\begin{align}
&\lambda_{\vvec z^1}(10,10)=1, 
\nonumber \\
&\lambda_{\vvec z^2}(10,10)=20, 
\qquad
\lambda_{\vvec z^2}(9,11)=18, 
\qquad
\lambda_{\vvec z^2}(9,10)=18, 
\label{exdata} \\
&\lambda_{\vvec z^3}(8,11)=33, 
\qquad\;\, \lambda_{\vvec z^3}(8,10)=30, 
\qquad \lambda_{\vvec z^3}(7,11)=33.
\nonumber
\end{align}
Using Lemma~$\ref{lem:find polynomial}$, we first find a CTMC 
$\overline{\vvec X}(t)$ with polynomial transition rates. 
Using the notation of Lemma~$\ref{lem:find polynomial}$ for
the first transition vector $\vvec z^1$, we
have $A_{\vvec z^1}=\{(10,10)\}$ such that $n=1=|\mathbb S_0|$,
matrix $M$ is scalar $M=1$ and `vector' $\vvec b$ 
is a scalar as well, $\vvec b=\lambda_{\vvec z^1}(10,10)=1$ .
Thus the polynomial transition 
rate $\lambda_{\vvec z^1}$ is a constant given by~$(\ref{cMb})$
as $\lambda_{\vvec z^1}=M^{-1}\vvec b=1$.
Considering transition vectors $\vvec z^2$ and $\vvec z^3$,
we have
$$
A_{\vvec z^2}=\{(10,10), \, (9,11), \, (9,10)\},
\quad
\mbox{and}
\quad
A_{\vvec z^3}=\{(8,11), \, (8,10), \, (7,11) \}.
$$
Since $|\mathbb S_1|=3$,  we find linear transition rate 
$\lambda_{\vvec z^2}$ (resp. $\lambda_{\vvec z^3}$)
of $\overline{\vvec X}(t)$ that have the values~$(\ref{exdata})$ 
at $A_{\vvec z^2}$ (resp. $A_{\vvec z^3}$). 
The $3\times 3$ matrix $M$ is given as
$$
M
=
\begin{bmatrix}
1 & 10 & 10\\
1 & 11 & 9\\
1 & 10 & 9
\end{bmatrix},
\qquad 
\mbox{respectively,} \qquad
M
=
\begin{bmatrix}
1 & 11 & 8 \\
1 & 10 & 8 \\
1 & 11 & 7
\end{bmatrix}.
$$
Since both matrices are invertible, we can calculate $\vvec c$ by
$(\ref{cMb})$, where $\vvec b = (20, 18, 18)^\top$, respectively
$\vvec b = (33, 30, 33)^\top.$ We obtain 
$\vvec c = M^{-1} \vvec b = (0,0,2)^\top$ for the transition
vector $\vvec z^2$ and 
$\vvec c = M^{-1} \vvec b = (0,3,0)^\top$ for the transition
vector  $\vvec z^3.$ Therefore, we obtain
$$
\lambda_{\vvec z^1} = 1,
\qquad
\lambda_{\vvec z^2}(\vvec x)=2x_1,
\qquad
\lambda_{\vvec z^3}(\vvec x)=3x_2.
$$
Next, we find a reaction network whose associated mass-action dynamics 
is equal to the CTMC $\overline{\vvec X}(t)$. 
The conditions~$(\ref{eq:poly decompose})$ and~$(\ref{eq:polynomial condition})$
of Theorem~$\ref{thm:polynomial and reaction network}$ are satisfied
for all three transition vectors $\vvec z^1$, $\vvec z^2$ and 
$\vvec z^3$. Thus the unique reaction system is
$$
\emptyset 
\; \mathop{\longrightarrow}^1 \;
X_1
\; \mathop{\longrightarrow}^2 \;
X_2
\; \mathop{\longrightarrow}^3 \;
\emptyset \, .
$$
\end{example}

\begin{example}
Consider the reaction system~$(\ref{eq:example network1})$ introduced 
in Example~$\ref{ex:non unique1}$. Let $\vvec z=(1,-1)^\top$ be one 
of the two transition vectors of the CTMC $\vvec X(t)$. Given the
transition rates~$(\ref{extransitionrates})$ on $\mathbb S_{\vvec v,2}$,
the first order reaction $X_2\to X_1$ is not identified 
using Theorem~$\ref{thm:polynomial and reaction network}$,
because matrix $M$ associated with states $\mathbb S_{\vvec v,2}$
is the singular matrix
$$
M
=
\begin{pmatrix}
1 & 2 & 0\\
1 & 1 & 1\\
1 & 0 & 2
\end{pmatrix}.
$$
\end{example}

\section{Inference of Reaction Networks using Temporal Data}

Theorem~\ref{thm:theoretical network inference} states that we can 
use transition rates and transition vectors of a mass-action stochastic
reaction system to uncover the underlying network structure. 
However, in applications, we are not given directly the transition
rates but temporal data consisting of states and transition times 
between them. For example, for an (a priori unknown) underlying 
network
$$
X_1 \, \xrightarrow{\;1\;} \, 2X_1, 
\qquad 
X_1+X_2 \, \xrightarrow{\;1\;} \, 2X_2,
$$
we are given transition data of the associated CTMC $\vvec X(t)$ such as 
\begin{align*}
\vvec X(0)=(1,1), \ \vvec X(\tau_1)=(2,1),  
\ \vvec X(\tau_2)=(1,2), \dots, 
\quad \text{and} \quad \tau_1=0.2, \ \tau_2=1.1,\dots,    
\end{align*}
where $\tau_i$ is the $i$-th transition time. Thus, to apply results
of the previous section, we need to use such time series to estimate
the transition vectors $(1,0)^\top$ and $(-1,1)^\top$ and
the corresponding transition rates 
$\lambda_{(1,0)}(\vvec x)= x_1$ and 
$\lambda_{(-1,1)}(\vvec x)= x_1 x_2$.

Suppose that we are given $Q$ sample trajectories of the CTMC 
$\vvec X(t)$ consisting of the states of the system 
$\vvec X^i(\tau_k^i)$, for $i=1,2,\dots,Q$, recorded
at times $\tau_k^i$, where $k=1,2,\dots,q(i),$ and $q(i)$ 
denotes the number of time points in the $i$-th time series.
Assuming that the given time series includes all reaction events,
the time of the $k$-th transition of the CTMC $\vvec X^i(t)$
is equal to $\tau_k^i$. Then all possible transition vectors $\vvec z$ 
of the system can be uncovered (as long as they are present 
in the recorded time series) by collecting the transitions 
$\vvec X^i(\tau_{k+1}^i)-\vvec X^i(\tau_k^i)$ for all 
$k=1,2,\dots,q(i)$ and $i=1,2,\dots,Q.$ 

Next, we estimate the transition rates at each state $\vvec x$ by 
using the sample trajectories. Let CTMC $\vvec X(t)$ be associated with 
reaction system $(\Rx, \Ki)$ and let $\cal Z$ be the finite set 
of transition vectors. Then by using the random time
representation~\cite{Kurtz:1972:RBS,Anderson:2015:SAB},
we have
$$
\vvec X(t)
=
\vvec X(0)
+
\sum_{\vvec z \in \cal Z} 
Y_{\vvec z} 
\left( 
\int_0^t  \lambda_{\vvec z}(\vvec X(s)) ds \right ) 
\vvec z,
$$
where $Y_{\vvec z}$ are independent unit Poisson processes. Therefore
$$
{\mathbb E}(\tau_{k+1} \ | \ \vvec X(\tau_{k}) = \vvec x)
=
\frac{1}{\lambda(\vvec x)},\quad 
\text{where} 
\quad
\lambda(\vvec x)=\sum_{\vvec z\in \cal Z} \lambda_{\vvec z}(\vvec x), 
$$
\vskip -2mm
\noindent 
and
\vskip -4mm
\begin{equation}
P \left(
\vvec X(\tau_{k+1})
= \vvec x+\vvec z\ 
\big| \ \vvec X(\tau_{k})=\vvec x \right)
=
\dfrac{\lambda_\vvec z(\vvec x)}{\lambda(\vvec x)}.
\label{lambdaz}
\end{equation}
To estimate $\lambda_{\vvec z}(\vvec x)$ at each state $\vvec x$,
we identify the data points when this state was reached by defining 
$G_{\vvec x}=\{ (i,k) \,|\, \vvec X^i (\tau_k^i) =\vvec x
\mbox{ and } k < q(i) \}$. 
Then, for each state $\vvec x$ and for each transition vector
$\vvec z$, we use
\begin{align}
\lambda_\vvec z(\vvec x)
=
\frac{\lambda_\vvec z (\vvec x)}{\lambda(\vvec x)}\lambda(\vvec x)
&=
\frac{
P \left(
\vvec X(\tau_{k+1})
= \vvec x+\vvec z\ 
\big| \ \vvec X(\tau_{k})=\vvec x \right)
}{
{\mathbb E}(\tau_{k+1} \ | \ \vvec X(\tau_{k}) = \vvec x)
}
\nonumber \\
&\approx \frac{\sum_{(i,k) \in G_{\vvec x}}
\mathbbm{1}_{\{\vvec X^i(\tau_{k+1}^i)-\vvec X^i(\tau_{k}^i) = \vvec z\}}}
{\sum_{(i,k)\in G_{\vvec x}} \tau^i_{k+1}},
\label{eq:estimate intensity}
\end{align}
where we assume that $|G_{\vvec x}|$ is sufficiently large to get
a good approximation. 

\begin{example}
Let $(\Rx,\Ki)$ be the following one-species mass-action reaction system,
\begin{align*}
3 X_1 
\; \xrightarrow{\;3\;} \;
4 X_1 
\; \xrightarrow{\;10\;} \;
\emptyset 
\; \xrightarrow{\;1\;} \;
X_1 
\; \xrightarrow{\;2\;} \;
2X_1.
\end{align*}
For the transition `vector' $\vvec z=1$, the transition rate of 
the associated CTMC $\vvec X(t)$ is 
$$
\lambda_\vvec z(x_1)
= 
1 + 2 x_1 + 3 x_1 (x_1-1)(x_1-2).
$$ 
Using the Gillespie SSA, we generate $Q = 10^2$ independent 
sample time trajectories 
of this system each of which contains $q(i)=10^3$ transition 
times $\tau^i_k$ and the corresponding states
$X_1^i(\tau^i_k)$, for $k=1,2,\dots,10^3$ and $i=1,2,\dots,10^2.$
Applying~$(\ref{eq:estimate intensity})$, we obtain for the state
$x_1 = 4$ the estimated transition rate $\lambda_{\vvec z}(4) = 80.871$, 
which compares well with the true transition rate 
$\lambda_{\vvec z}(4)=81$.
\end{example}

\subsection{Distance between two reaction systems}

For a given (unknown) mass-action reaction system $(\Rx,\Ki)$, suppose 
we know the number of species and the order of the network. 
Suppose further that we use transition data associated with $(\Rx,\Ki)$ 
to estimate the transition rates of $(\Rx,\Ki)$ by 
equation~(\ref{eq:estimate intensity}). Then we can use the 
estimated transition rates to infer a reaction system 
$(\overline \Rx, \overline \Ki)$ by 
applying Theorem~\ref{thm:theoretical network inference}. 
In this section, we discuss how we can measure the accuracy of the 
inferred reaction system $(\overline \Rx, \overline \Ki)$ by 
comparing to the original system $(\Rx,\Ki)$.

\begin{definition}
For two reaction systems $(\Rx,\Ki)$ and $(\overline \Rx,\overline \Ki)$ 
defined on $\mathbb Z^d_{\ge 0}$, their distance
at time $t$ is defined as the total variance distance as
$\Vert p(\cdot,t)-\overline{p}(\cdot,t)\Vert_{TV},$
where $p(\vvec x,t) = P(\vvec X(t) = \vvec x)$ 
and $\overline{p}(\vvec x,t)=P(\overline{\vvec X}(t) = \vvec x)$ 
are the probability distributions of the stochastic systems $\vvec X(t)$ 
and $\overline{\vvec X}$ associated with $(\Rx,\Ki)$ 
and $(\overline \Rx,\overline \Ki)$, respectively.
In particular, we measure the similarity of the two reaction systems on 
a finite set $U$ with their distance at time $t$ with respect 
to a finite set $U$, which we define as
\begin{align*}
\delta_{U}=
\frac{1}{2}
\sum_{\vvec x\in U} \big|p(\vvec x,t)-\overline{p}(\vvec x,t) \big|.
\end{align*}
\end{definition}

\noindent
An alternative distance can also be defined by measuring the difference
between the reaction intensities of $(\Rx,\Ki)$ and 
$(\overline\Rx,\overline \Ki)$ over a fixed finite set.
\begin{definition}
For two reaction systems $(\Rx,\Ki)$ and $(\overline \Rx,\overline \Ki)$ 
defined on $\mathbb Z^d_{\ge 0}$, let $\vvec X(t)$ and $\vvec X(t)$ be 
the associated CTMCs with the set of transition vectors $\cal Z$ and
$\overline{\cal Z}$, respectively. Let further that $\lambda_{\vvec z}$ 
and $\overline \lambda_{\bar{\vvec z}}$ be the transition rates associated
with transition vectors $\vvec z\in \cal Z$ and 
$\bar{\vvec z} \in \cal \overline{\cal Z}$, respectively.
Then for a fixed finite set $U$, we define
$$
\delta^I_U
=
\max_{\vvec x\in U} 
\left \{ \max_{\vvec z\in \cal Z\cap \overline{\cal Z}}
\left|
\lambda_\vvec z(\vvec x)-\overline \lambda_\vvec z(\vvec x)
\right|, \;
\max_{\vvec z\in \cal Z} \lambda_\vvec z(\vvec x), 
\;
\max_{\bar{\vvec z}\in \overline{\cal Z}}  
\overline \lambda_{\bar{\vvec z}}(\vvec x) 
\right \}.
$$
\end{definition}

\noindent
Both the distances $\delta_U$ and $\delta^I_U$ measure the similarity of 
two reaction systems confined to a finite set $U$.
For a given (unknown) reaction system of order $N$, 
we can apply Theorem~\ref{thm:theoretical network inference} to infer 
a network system by using the transition data over $U=\mathbb S_N$. 
Then we can test with either $\delta_U$ or $\delta^I_U$ 
how close the inferred network is to the original reaction system. 
The following example demonstrates this process.

\begin{example}
\label{ex:gillespie}
Consider the following mass-action reaction system of order $3$:
\begin{equation}
\begin{split}
& 
\qquad \qquad
X_1 \; \xrightleftarrows{1}{1}{4 mm} \; \emptyset,
\qquad
X_2 \; \xrightleftarrows{1}{1}{4 mm} \; \emptyset,
\\
&2 X_1 + X_2 
\; \xrightarrow{\;\; 1 \;\;} \; 
\emptyset
\; \xrightarrow{\;\; 1 \;\;} \; 
X_1 + X_2
\; \xrightarrow{\;\; 1 \;\;} \;
2 X_1 + 2 X_2
\end{split}
\label{exorig}
\end{equation}
We use the Gillespie SSA to simulate
the reaction system~$(\ref{exorig})$ until we collect 
$15.625\times 10^5$ sample transition times $\tau^i_k$ for 
each state $\vvec x^j \in \mathbb S_3$, where $j=1,2,\dots,10$. 
Then we estimate the transition rates by~$(\ref{eq:estimate intensity})$
and apply Theorem~$\ref{thm:theoretical network inference}$ 
with the estimated transition rates over $\mathbb S_3$. 
We obtain the mass-action reaction system which contain both
original reactions (with modified rate constants)
\begin{equation}
\begin{split}
& 
\qquad \qquad \quad
X_1 \; \xrightleftarrows{0.9999}{1.0008}{10 mm} \; \emptyset,
\qquad
X_2 \; \xrightleftarrows{1.0025}{0.9996}{10 mm} \; \emptyset,
\\
&2 X_1 + X_2 
\; \xrightarrow{\;\; 0.9994 \;\;} \; 
\emptyset
\; \xrightarrow{\;\; 1.0002 \;\;} \; 
X_1 + X_2
\; \xrightarrow{\;\; 1.0027 \;\;} \;
2 X_1 + 2 X_2,
\end{split}
\label{part1network}
\end{equation}
and additional reactions (with relatively small rate constants)
\begin{equation}
\begin{split}
&
2 X_1 + X_2 
\; \xrightarrow{\;\; 0.0022 \;\;} \;
X_1 + X_2
\; \xrightarrow{\;\; 0.0013 \;\;} \; 
X_2
\; \xrightarrow{\;\; 1.3 \times 10^{-4} \;\;} \; 
2 X_2
\; \xrightarrow{\;\; 5.3 \times 10^{-4} \;\;} \; 
3 X_2
\\
&
\qquad \qquad
2 X_1
\; \xrightarrow{\;\; 5.2 \times 10^{-4} \;\;} \; 
X_1,
\qquad
2 X_1 
\; \xrightarrow{\;\; 4.9 \times 10^{-4} \;\;} \;
2 X_1 + X_2,
\end{split}
\label{part2network}
\end{equation}
where the reactions in $(\ref{part2network})$ are
the reactions in $\overline \Rx\setminus \Rx$. 
To compare the original reaction system~$(\ref{exorig})$ with
the inferred reaction system~$(\ref{part1network})$--$(\ref{part2network})$,
we first estimate the distance $\delta_U$ by computing the empirical 
measures with $10^4$ realisations of the Gillespie SSA. We obtain 
$\delta_U=0.0083$ (for a larger set $U'=\mathbb S_{100}$, we get 
$\delta_{U'}=0.0244$). The alternative distance $\delta^I_U$ can 
also be computed using the mass-action intensities of the reaction 
systems as $\delta^I_{U}=0.0090$ (for the larger set 
$U'=\mathbb S_{100}$, we get $\delta^I_{U'}=331.9268$).
Mean trajectories of species $X_1$ and $X_2$ in the original 
reaction system~$(\ref{exorig})$ and the inferred reaction
network~$(\ref{part1network})$--$(\ref{part2network})$ are shown 
in Figure~$\ref{fig1}$.

\begin{figure}[tp]
\leftline{(a) \hskip 7.9cm (b)}
\centerline{\hskip 8mm \includegraphics[height=6.2cm]{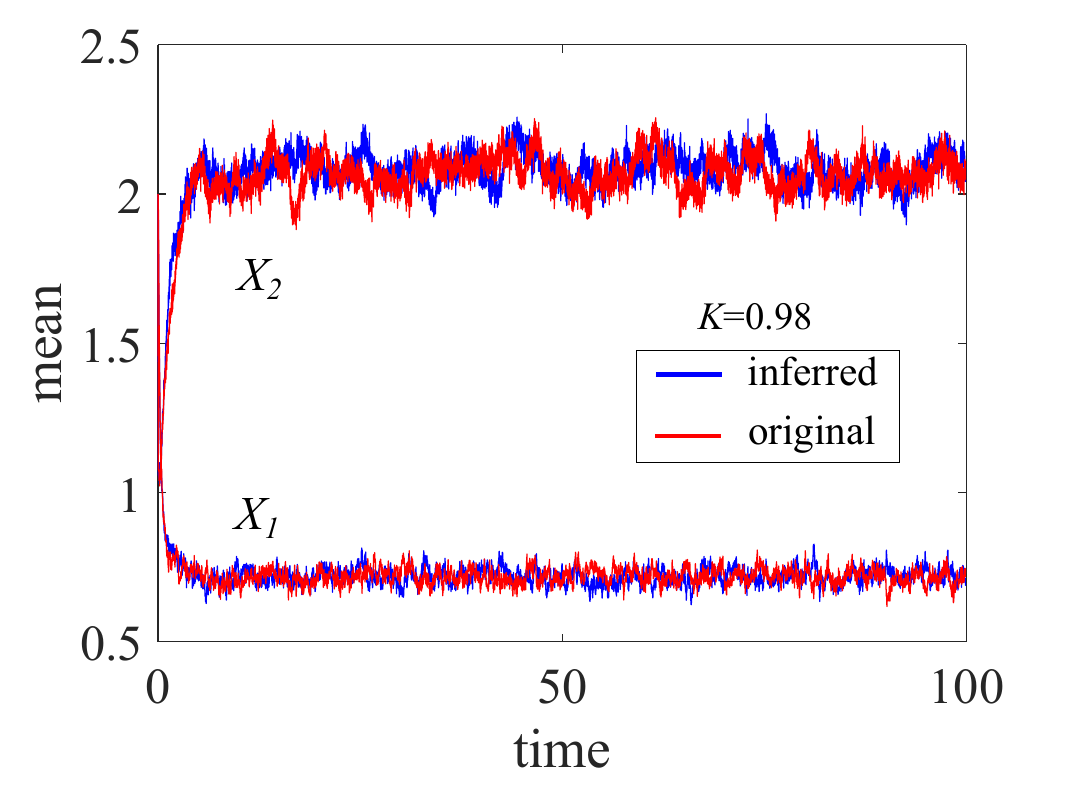}
\includegraphics[height=6.2cm]{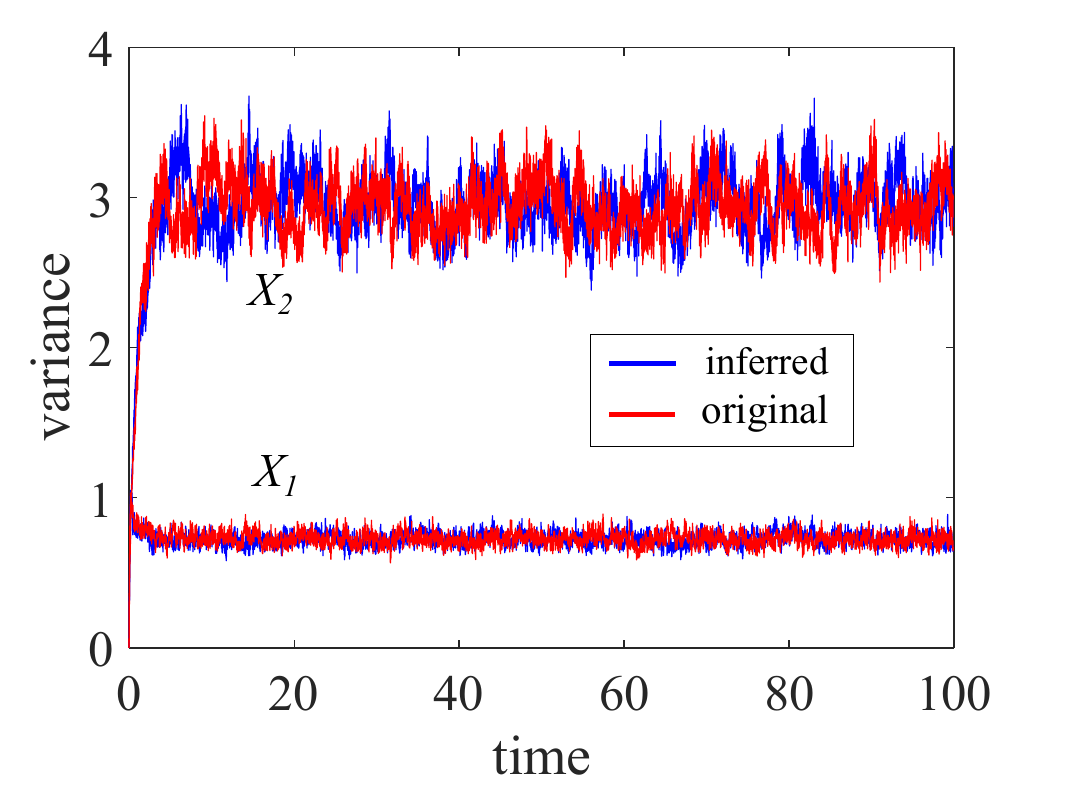}}
\caption{(a)
{\it Mean values of $X_1(t)$ and $X_2(t)$ of the reaction 
network~$(\ref{exorig})$ and the inferred reaction
network~$(\ref{part1network})$--$(\ref{part2network})$
obtained by averaging over $10^4$ realisations of the 
Gillespie SSA with initial condition $X_1(0)=1$ and
$X_2(0)=1$. The average number of transitions by 
the reactions~$(\ref{part2network})$ in 
$\overline \Rx\setminus \Rx$ is denoted by~$K$.} \\
(b) {\it The variance of $X_1(t)$ and $X_2(t)$ estimated
from the same time series.}
}
\label{fig1}
\end{figure}
\end{example}

\begin{remark}
As shown in Example~$\ref{ex:gillespie}$, the distance $\delta_U$ is 
robust to the size of $U$ because this distance is defined using 
the probability densities. However, the distance $\delta^I_U$ is 
sensitive to the choice of the set $U$ since the transition rates
$\lambda_{\vvec z}(\vvec x)$ and 
$\overline \lambda_{\bar{\vvec z}}(\vvec x)$ rapidly increase 
as $\Vert \vvec x \Vert_1$ is increased. 
\end{remark}

\subsection{Error Analysis}

For a given CTMC, the true underlying network structure and the true 
parameter values are often unknown. Thus the distance between the 
true network and the estimated network cannot be calculated. 
By using the central limit theorem, however, we can find confidence 
intervals for given stochastic simulation data to ensure that the 
alternative distance $\delta^I_{U}$ is less than some bound. 
Let $(\Rx,\Ki)$ be a given reaction system and let 
$\lambda_\vvec z(\vvec x)$ be the transition rate of the associated CTMC.
Note that 
$$
\lambda_\vvec z(\vvec x)
=
P \left(
\vvec X(\tau_{k+1})
= \vvec x+\vvec z\ 
\big| \ \vvec X(\tau_{k})=\vvec x \right)
\,
\lambda(\vvec x)
$$ as shown in \eqref{lambdaz}, where 
$\lambda(\vvec x) = \sum_{\vvec z\in \cal Z}\lambda_\vvec z(\vvec x)$ 
is the total intensity of the CTMC $\vvec X(t)$. 
Thus letting 
$
\overline \lambda(\vvec x)
=
|G_\vvec x| 
/
\left( \sum_{(i,k)\in G_\vvec x} \tau^i_{k+1} \right )$ 
be the sample mean of the total intensity, we define the sample 
transition rate for a transition vector $\vvec z$ as 
$$
\overline \lambda^{(i,k)}_\vvec z(\vvec x)
=
\mathbbm{1}_{\{\vvec X^i(\tau_{k+1}^i)-\vvec X^i(\tau_{k}^i) = \vvec z\}}
\, \overline \lambda(\vvec x)
.$$ 
Then the sample mean of the transition rate 
$\overline \lambda_\vvec z(\vvec x)$ can be computed as 
$$
\overline \lambda_\vvec z(\vvec x)
=
\frac{1}{|G_\vvec x|}\sum_{(i,k)\in G_{\vvec x}}
\mathbbm{1}_{\{\vvec X^i(\tau_{k+1}^i)-\vvec X^i(\tau_{k}^i) = \vvec z\}}
\, {\lambda}(\vvec x) 
\approx  
\frac{1}{|G_\vvec x|}\sum_{(i,k)\in G_{\vvec x}} 
\overline \lambda^{(i,k)}_\vvec z(\vvec x).
$$
Then by the central limit theorem, for $\varepsilon>0$
$$
P\left ( \lambda_\vvec z(\vvec x)-\varepsilon 
\le 
\overline \lambda_\vvec z(\vvec x) 
\le  \lambda_\vvec z(\vvec x)+\varepsilon \right )
\approx 
P\left (-\frac{\varepsilon\sqrt{|G_\vvec x|}}{\sigma_\vvec z(\vvec x)} 
\le 
Z 
\le \frac{\varepsilon\sqrt{|G_\vvec x|}}{\sigma_\vvec z(\vvec x)} \right ),
$$
where 
$$
\sigma^2_\vvec z(\vvec x)
=
\dfrac{1}{|G_\vvec x|-1}
\sum_{(i,k)\in G_\vvec x}
\left (\overline \lambda^{(i,k)}_\vvec z(\vvec x)
-\overline \lambda_\vvec z (\vvec x)\right )^2
$$ 
is the sample variance, and $Z$ is an independent standard normal 
random variable. Thus we can formulate the following proposition
on confidence intervals.

\begin{proposition}
Let $(\Rx,\Ki)$ be a reaction system. For a finite subset 
$A \subseteq \mathbb Z^d_{\ge 0}$, let 
$\overline \lambda_\vvec z(\vvec x)$ 
and $\sigma^2_\vvec z(\vvec x)$ be the sample mean and the sample 
variance for each transition vector $\vvec z \in \cal Z$ and 
$\vvec x\in A$, respectively. For some $0<\alpha<1$, suppose 
that $\varepsilon>0$ satisfies
\begin{equation}
\varepsilon
\ge 
\dfrac{z_\alpha \sigma_\vvec z(\vvec x)}{\sqrt{|G_\vvec x|}} 
\quad \text{for each $\vvec z \in \cal Z$ and for each $\vvec x \in A$},
\end{equation}
where $[-z_\alpha,z_\alpha]$ is the $(1-\alpha)$-confidence interval of 
a standard normal random variable, i.e.
$P(-z_\alpha \le Z \le z_\alpha)=1-\alpha$, where $Z$ is the standard
normal random variable. Then for the inferred reaction system 
$(\overline \Rx,\overline \Ki)$ obtained by 
Theorem~$\ref{thm:theoretical network inference}$ with the 
sample transition rates $\overline \lambda_{\vvec z}(\vvec x)$, 
the distance $\delta^I_U$ between 
$(\Rx,\Ki)$ and $(\overline \Rx,\overline \Ki)$ is less 
than $\varepsilon$ with $(1-\alpha) \times 100 \%$ accuracy. 
\end{proposition}

\begin{example}
Consider again the inferred reaction
system~$(\ref{part1network})$--$(\ref{part2network})$ 
in Example~$\ref{ex:gillespie}$. Note that we have the sample 
transition rates $\overline \lambda_{\vvec z}(\vvec x)$ at 
each state $\vvec x\in \mathbb S_3$ and for each transition 
vector $\vvec z$. Hence we can calculate the sample variance.
We obtain
\begin{center}
\begin{tabular}{ c c c c c c } 
\hline \\ [-0.3cm]
$\vvec x$ & $(0,0)$ & $(1,0)$ & $(2,1)$ & $\dots$ & $(3,0)$ \\[1ex]
\hline  \\ [-0.1cm]
$\vvec z$ & $(1,1)^\top$ & $(1,1)^\top$ & $(-1,0)^\top$ & $\dots$ 
& $(-1,0)^\top$ \\[1ex]
\hline \\ [-0.1cm]
$\overline \lambda_\vvec z (\vvec x)$ & $1.002$ & $0.9984$ & $2.0073$ & 
$\dots$ & $2.9975$ \\[1ex]
\hline \\ [-0.1cm]
$\sigma_\vvec z(\vvec x)$ & $2.007$ & $2.9966$ & $8.0045$ & $\dots$ & 
$2.9972$ \\[1ex]
\hline
\end{tabular}
\end{center}
For $\alpha=0.05$, we have $z_\alpha=1.96$. Hence if we 
let 
$
\varepsilon
=
0.0141
=
\displaystyle\max_{\vvec x \in \mathbb S_3,\vvec z \in \cal Z} 
z_\alpha \sigma_\vvec z(\vvec x) |G_\vvec x|^{-1/2}$, then
the distance $\delta^I_{\mathbb S_3}$ between the given system and 
the estimated reaction system is less than $0.0141$ with 
$95\%$ accuracy.
\end{example}

\section{Discussion}

In this paper we have explored identifiability of reaction systems. 
Identifiability of a stochastic reaction system $(\Rx,\Ki)$ holds if this is 
the only set of reactions that produces its transition rates on the
corresponding state space. Therefore identifiability of a reaction system 
must be verified \textit{prior to} inference of a network structure and 
parameter estimation. By using the fact that a mass-action system is 
fully characterized with the transition rates on a certain finite region, 
we proved that any stochastic mass-action system of order at most $N$ 
is identifiable as long as the associated state space contains~$\mathbb S_{N}$.

By using the mass-action property, we have also proposed an algorithm 
that enables us to infer the underlying reaction network and the associated 
parameters with the transition data of a given CTMC. In the case that the 
transition data are given by stochastic simulations, we have investigated 
how to approximate the true transition data, and in turn, how to infer 
an estimated underlying network. Then by using the confidence intervals, 
we can measure the accuracy of the estimated underlying network comparing 
to the true network.

The presented network inference method relies on the exact transition data 
consisting of the transition vectors and the transition times. Hence 
our method is not directly applicable to data that consists of partial
information of the system at discrete time points. However we have shown 
that as the transition information and confidence on transition rate 
estimates increases, the distance between the actual and approximated 
networks tends to decrease. Given that increasingly precise measurements 
are being made for specific reaction networks in experimental studies, 
we expect that our method can be used in the future to infer underlying
networks and kinetic parameters for realistic biological systems.  

\section*{Acknowledgment}
Radek Erban and German Enciso would like to thank the organizers of the
``Recent Developments in Mathematical and Computational Biomedicine" 
(19w5085) workshop at the Casa Matem\'atica Oaxaca (CMO) in Oaxaca, in
November 2019, where this research project was initiated. German Enciso 
and Jinsu Kim are partially supported by NSF grant DMS1763272, Simons
Foundation grant 594598 (Qing Nie) and by NSF grant DMS1616233.


\begin{thebibliography}{10}

\bibitem{Anderson:2011:DAB}
D.~Anderson and T.~Kurtz.
\newblock {Continuous time Markov chain models for chemical reaction networks}.
\newblock In H.~Koeppl, editor, {\em Design and Analysis of Biomolecular
  Circuits: Engineering Approaches to Systems and Synthetic Biology}, pages
  3--42. Springer, 2011.

\bibitem{Anderson:2015:SAB}
D.~Anderson and T.~Kurtz.
\newblock {\em Stochastic Analysis of Biochemical Systems}.
\newblock Springer, 2015.

\bibitem{Angeli:2009:TCR}
D.~Angeli.
\newblock A tutorial on chemical reaction network dynamics.
\newblock {\em European Journal of Control}, 15:398 -- 406, 2009.

\bibitem{baldi2001bioinformatics}
P.~Baldi and S.~Brunak.
\newblock {\em Bioinformatics: the machine learning approach}.
\newblock MIT press, 2001.

\bibitem{catanach2020bayesian}
T.~Catanach, H.~Vo, and B.~Munsky.
\newblock Bayesian inference of stochastic reaction networks using
  multifidelity sequential tempered {M}arkov chain {M}onte {C}arlo.
\newblock {\em arXiv preprint arXiv:2001.01373}, 2020.

\bibitem{chattopadhyay2013inverse}
I.~Chattopadhyay, A.~Kuchina, G.~S{\"u}el, and H.~Lipson.
\newblock Inverse {G}illespie for inferring stochastic reaction mechanisms from
  intermittent samples.
\newblock {\em Proceedings of the National Academy of Sciences},
  110(32):12990--12995, 2013.

\bibitem{Craciun:2006:MEC}
G.~Craciun and M.~Feinberg.
\newblock Multiple equilibria in complex chemical reaction networks: Ii. the
  species-reactions graph.
\newblock {\em SIAM Journal on Applied Mathematics}, 66(4):1321--1338, 2006.

\bibitem{craciun2013statistical}
G.~Craciun, J.~Kim, C.~Pantea, and G.~Rempala.
\newblock Statistical model for biochemical network inference.
\newblock {\em Communications in Statistics-Simulation and Computation},
  42(1):121--137, 2013.

\bibitem{Craciun:2008:ICR}
G.~Craciun and C.~Pantea.
\newblock Identifiability of chemical reaction networks.
\newblock {\em Journal of Mathematical Chemistry}, 44:244--259, 2008.

\bibitem{Duncan:2015:NIM}
A.~Duncan, S.~Liao, T.~Vejchodsk\'y, R.~Erban, and R.~Grima.
\newblock Noise-induced multistability in chemical systems: Discrete versus
  continuum modeling.
\newblock {\em Physical Review E}, 91:042111, Apr 2015.

\bibitem{Erban:2009:ASC}
R.~Erban, S.~J. Chapman, I.~Kevrekidis, and T.~Vejchodsky.
\newblock Analysis of a stochastic chemical system close to a {SNIPER}
  bifurcation of its mean-field model.
\newblock {\em SIAM Journal on Applied Mathematics}, 70(3):984--1016, 2009.

\bibitem{Erban:2020:SMR}
R.~Erban and S.J. Chapman.
\newblock {\em Stochastic Modelling of Reaction–diffusion Processes}.
\newblock Cambridge University Press, 2020.

\bibitem{Feinberg:1989:NSC}
M~Feinberg.
\newblock Necessary and sufficient conditions for detailed balancing in mass
  action systems of arbitrary complexity.
\newblock {\em Chemical Engineering Science}, 44(9):1819--1827, 1989.

\bibitem{Feinberg:2019:FCR}
M.~Feinberg.
\newblock {\em Foundations of Chemical Reaction Network Theory}.
\newblock Springer, 2019.

\bibitem{Gadgil:2005:SAF}
C.~Gadgil, C.~Lee, and H.~Othmer.
\newblock A stochastic analysis of first-order reaction networks.
\newblock {\em Bulletin of Mathematical Biology}, 67:901--946, 2005.

\bibitem{golightly2006bayesian}
A.~Golightly and D.~Wilkinson.
\newblock Bayesian sequential inference for stochastic kinetic biochemical
  network models.
\newblock {\em Journal of Computational Biology}, 13(3):838--851, 2006.

\bibitem{gupta2014comparison}
A.~Gupta and J.~Rawlings.
\newblock Comparison of parameter estimation methods in stochastic chemical
  kinetic models: examples in systems biology.
\newblock {\em AIChE Journal}, 60(4):1253--1268, 2014.

\bibitem{Jahnke:2007:SCM}
T.~Jahnke and W.~Huisinga.
\newblock Solving the chemical master equation for monomolecular reaction
  systems analytically.
\newblock {\em Journal of Mathematical Biology}, 54(1):1--26, 2007.

\bibitem{Komorowski:2011:SRA}
M.~Komorowski, M.~Costa, D.~Rand, and M.~Stumpf.
\newblock Sensitivity, robustness, and identifiability in stochastic chemical
  kinetics models.
\newblock {\em Proceedings of the National Academy of Sciences},
  108(21):8645--8650, 2011.

\bibitem{Kurtz:1972:RBS}
T.~Kurtz.
\newblock {The relationship between stochastic and deterministic models for
  chemical reactions}.
\newblock {\em Journal of Chemical Physics}, 57(7):2976--2978, 1972.

\bibitem{Langary:2019:ICR}
D.~Langary and Z.~Nikoloski.
\newblock Inference of chemical reaction networks based on concentration
  profiles using an optimization framework.
\newblock {\em Chaos: An Interdisciplinary Journal of Nonlinear Science},
  29(11):113121, 2019.

\bibitem{Liao:2015:TMP}
S.~Liao, T.~Vejchodsk{\'y}, and R.~Erban.
\newblock Tensor methods for parameter estimation and bifurcation analysis of
  stochastic reaction networks.
\newblock {\em Journal of the Royal Society Interface}, 12(108):20150233, 2015.

\bibitem{Loskot:2019:CRM}
P.~Loskot, K.~Atitey, and L.~Mihaylova.
\newblock Comprehensive review of models and methods for inferences in
  bio-chemical reaction networks.
\newblock {\em Frontiers in Genetics}, 10:549, 2019.

\bibitem{markowetz2007inferring}
F.~Markowetz and R.~Spang.
\newblock Inferring cellular networks--a review.
\newblock {\em BMC Bioinformatics}, 8(6):S5, 2007.

\bibitem{Plesa:2019:NIM}
T.~Plesa, R.~Erban, and H.~Othmer.
\newblock Noice-induced mixing and multimodality in reaction networks.
\newblock {\em European Journal of Applied Mathematics}, 30:887--911, 2019.

\bibitem{Plesa:2016:CRS}
T.~Plesa, T.~Vejchodsk\'y, and R.~Erban.
\newblock {C}hemical reaction systems with a homoclinic bifurcation: an inverse
  problem.
\newblock {\em Journal of Mathematical Chemistry}, 54(10):1884--1915, 2016.

\bibitem{Plesa:2017:TMS}
T.~Plesa, T.~Vejchodsk{\'y}, and R.~Erban.
\newblock Test models for statistical inference: Two-dimensional reaction
  systems displaying limit cycle bifurcations and bistability.
\newblock In {\em Stochastic Processes, Multiscale Modeling, and Numerical
  Methods for Computational Cellular Biology}, pages 3--27. Springer
  International Publishing, 2017.

\bibitem{Plesa:2018:NCM}
T.~Plesa, K.~Zygalakis, D.~Anderson, and R.~Erban.
\newblock Noise control for molecular computing.
\newblock {\em Journal of the Royal Society Interface}, 15(144):20180199, 2018.

\bibitem{szederkenyi2011inference}
G.~Szederk{\'e}nyi, J.~Banga, and A.~Alonso.
\newblock Inference of complex biological networks: distinguishability issues
  and optimization-based solutions.
\newblock {\em BMC Systems Biology}, 5(1):177, 2011.

\bibitem{villaverde2014reverse}
A.~Villaverde and J.~Banga.
\newblock Reverse engineering and identification in systems biology:
  strategies, perspectives and challenges.
\newblock {\em Journal of the Royal Society Interface}, 11(91):20130505, 2014.

\bibitem{walter1997identification}
E.~Walter and L.~Pronzato.
\newblock Identification of parametric models.
\newblock {\em Communications and Control Engineering}, 8, 1997.

\bibitem{Wang:2019:IRN}
S.~Wang, J.~Lin, E.~Sontag, and P.~Sorger.
\newblock Inferring reaction network structure from single-cell, multiplex
  data, using toric systems theory.
\newblock {\em PLOS Computational Biology}, 15:1--25, 12 2019.

\bibitem{warne2019simulation}
D.~Warne, R.~Baker, and M.~Simpson.
\newblock Simulation and inference algorithms for stochastic biochemical
  reaction networks: from basic concepts to state-of-the-art.
\newblock {\em Journal of the Royal Society Interface}, 16(151):20180943, 2019.

\end{thebibliography}
\end{document}